\documentclass[11pt,letterpaper]{article}
\usepackage{arxiv_compat}
\usepackage[numbers,sort&compress]{natbib}
\def\newblock{\ }%
\bibpunct[, ]{[}{]}{,}{n}{}{,}%
\usepackage{chapterbib}
\usepackage[colorlinks=false,breaklinks=true,bookmarks=true,hidelinks]{hyperref}
\hypersetup{pdfauthor={Yingzhe Xu},pdftitle={When Can an Infinite-Stream Evaluation Support Dynamic Programming? Compatibility, Recursive Compression, and Generalized Summation}}

\theoremstyle{TH}
\newtheorem{theorem}{Theorem}[section]
\newtheorem{lemma}{Lemma}[section]
\newtheorem{proposition}{Proposition}[section]
\newtheorem{corollary}{Corollary}[section]
\theoremstyle{EX}

\newtheorem{remark}{Remark}[section]

{\theoremstyle{TH}\newtheorem*{keylemma}{Lemma}}
{\theoremstyle{TH}\newtheorem*{keyproposition}{Proposition}}
{\theoremstyle{TH}\newtheorem*{keycorollary}{Corollary}}
{\theoremstyle{EX}\newtheorem*{keydefinition}{Definition}}
{\theoremstyle{EX}\newtheorem*{keyremark}{Remark}}

\newcommand{\Nzero}{\mathbb{N}_0}

\newcommand{\Dclass}{\mathsf{D}}
\newcommand{\Fclass}{\mathsf{F}}
\newcommand{\Sclass}{\mathsf{S}}
\newcommand{\Oclass}{\mathsf{O}}
\newcommand{\Mclass}{\mathsf{M}}
\newcommand{\aut}{\mathrm{aut}}
\newcommand{\prodtop}{\tau_{\mathrm{prod}}}
\newcommand{\suptop}{\tau_{\infty}}
\newcommand{\BX}{\mathcal{B}_X}
\DeclareMathOperator{\Log}{Log}

\begin{document}
\TITLE{When Can an Infinite-Stream Evaluation Support Dynamic Programming? Compatibility, Recursive Compression, and Generalized Summation}
\ARTICLEAUTHORS{\AUTHOR{Yingzhe Xu}\AFF{The Australian National University}}
\ABSTRACT{Dynamic programming for nonstandard infinite-horizon criteria requires more than finite stream values: one-step control substitutions must remain evaluable, and continuation information must propagate recursively. For expected stage-reward streams, we separate domain admissibility, finite-dimensional continuous autonomous realizability, scalar sufficiency, and contextual order consistency. For absolutely summable weighted criteria, we prove that the minimum Euclidean dimension of an exact continuous autonomous continuation state equals the temporal Hankel rank. Thus finite Hankel rank exactly characterizes when temporal nonstationarity can be replaced by finite endogenous state augmentation; the lower bound applies even to nonlinear continuous encoders and already holds in a one-state MDP. For literal Candelpergher Ramanujan summation, the same common-continuation test fails earlier: changing only the current action can leave the evaluation domain. Abel summation provides a contrasting prefix-closed benchmark.}
\maketitle

\section{Introduction}
\label{sec:introduction}

Dynamic programming requires more than a finite value for every infinite reward stream. A current action must compose with a common continuation in a way that remains inside the evaluation domain, and the continuation information needed after one step must be recursively propagable. Standard discounted criteria build these properties into the objective; a general infinite-stream evaluation need not.

We study deterministic streams of \emph{expected stage rewards} generated by finite Markov decision processes. For an initial state $s$ and policy $\pi$, let
\[
g^\pi(s)=\bigl(\mathbb E_s^\pi[r(S_0,A_0)],\mathbb E_s^\pi[r(S_1,A_1)],\ldots\bigr),
\qquad
J_E^\pi(s)=E(g^\pi(s))
\]
whenever the possibly partial functional $E$ is defined at that stream. The object is therefore an evaluation of the expected-reward stream, not the full return distribution or an expectation of a nonlinear pathwise functional.

\subsection{Compatibility gates and main results}

The paper organizes Bellman compatibility into four successive questions. First, \emph{domain admissibility}: do the one-step prefixing and mixture operations generated by control stay inside the domain? Second, \emph{finite-dimensional continuous autonomous realizability}: can a common continuation be propagated exactly through a fixed Euclidean state without an external clock? Third, \emph{scalar sufficiency}: can the evaluation value itself serve as that state? Fourth, \emph{contextual order consistency}: is the scalar comparison preserved when the same continuation is inserted into a common control context? Stagewise monotonicity is tracked separately. On a rich bounded-stream class these conditions form a strict hierarchy.

Our first principal result gives an exact memory characterization for absolutely summable weighted criteria
\[
E_w(x)=\sum_{n\ge0}w_nx_n,
\qquad w\in\ell^1.
\]
Let $S$ be the coordinate left shift on the weight sequence and define
\[
r_w=\dim\operatorname{span}\{S^iw:i\ge0\},
\qquad
H_w=(w_{i+j})_{i,j\ge0}.
\]
For both the product and sup-norm topologies on $X=[0,1]^{\mathbb N_0}$, Theorem~\ref{thm:weighted-hankel-characterization} proves
\[
k_{\min}^{\mathrm{cont,aut}}(E_w)
=
r_w
=
\operatorname{rank}_{\mathrm{alg}}H_w.
\]
Finite rank yields an explicit affine residual-coordinate realization. Conversely, $k+1$ independent residual statistics produce a $(k+1)$-dimensional cube that exact autonomy would have to inject continuously into $\mathbb R^k$, contradicting invariance of domain. The lower bound therefore applies to arbitrary continuous nonlinear Euclidean encoders within the stated continuation architecture, not only to linear or state-affine realizations.

From a dynamic-programming perspective, this is an exact state-augmentation theorem. If calendar time is supplied externally, every fixed weight sequence admits a scalar time-indexed continuation value. Theorem~\ref{thm:weighted-hankel-characterization} identifies exactly when that exogenous clock can instead be eliminated and replaced by a fixed finite-dimensional endogenous state. Thus $r_w$ separates weighted criteria that admit exact finite autonomous augmentation from those for which, within the present continuous architecture, an exact formulation must retain an external clock or leave the finite-dimensional class. The lower bound already holds in the one-state two-action model of Lemma~\ref{lem:richness}, so the required memory is induced by the objective rather than by physical-state complexity.

Classical Hankel/formal-series realization theory supplies the finite-rank algebraic backbone \citep{schutzenberger1961automata,carlylepaz1971realizations,fliess1974hankel,sontag1979realization}. The new point is class-relative but stronger on competitor regularity: for the present continuation architecture, ordinary temporal row rank remains sharp even against arbitrary continuous nonlinear Euclidean encoders, with no linearity, differentiability, or analyticity assumed of the competing representation. Section~\ref{sec:discussion} gives the detailed comparison with nonlinear, topological, and information-state realization theories.

Our second principal result applies the same \emph{common-continuation substitution} one gate earlier. For literal Candelpergher Ramanujan summation \citep{candelpergher2017}, Carlson uniqueness makes the analytic sample domain tail-rigid. Theorem~\ref{thm:ram-domain-obstruction} shows that in a one-state two-action self-loop MDP the zero continuation is admissible, whereas changing only the current action produces $(1,0,0,\ldots)$, which is not. Thus the two action alternatives need not share an evaluation domain, so a Bellman state or backup is not yet well posed. This is a domain-level obstruction, not an infinite-memory claim: it arises before recursive compression or optimality questions are meaningful. Separate propositions show scalar-value insufficiency and failure of stagewise monotonicity.

These two results are different failure levels of one operation: keep the continuation fixed and vary the current prefix or branch. Ramanujan can fail before factorization is meaningful; total weighted criteria pass the domain gate, after which temporal Hankel rank gives the exact autonomous continuation-state dimension. A supporting summability result shows that finite-support regularity plus scalar prefix sufficiency forces the classical additive prefix law, while Abel summation supplies a positive generalized-summation benchmark.

\subsection{Positioning and scope}

Infinite-stream valuations and undiscounted MDP objectives already admit rich axiomatic and finitely additive theories \citep{jonsson2023axiomatic,neyman2023valuations,flesch2026finitelyadditive}. Related recursive-state constructions arise for Bellman-closed statistics, recursive preferences, and preference-augmented states \citep{rowland2019statistics,marthe2023beyond,haskell2026history,bommier2017monotone,marinacci2026recursive}. Our question is prior to optimal-control analysis: for a specified expected-stream functional, are one-step substitutions admissible, and what exact continuation state do they require? The autonomous qualifier excludes an externally supplied calendar clock; the lower bound does not exclude approximation, discontinuous encodings, or infinite-dimensional states.

The closest realization-theoretic comparison is deliberately narrower than a general nonlinear-realization claim. Classical one-variable Hankel theory identifies finite rank with recurrence/rationality and minimal linear or state-affine realizations, whereas nonlinear realization results use stronger regularity assumptions or different Hankel-type invariants. Theorem~\ref{thm:weighted-hankel-characterization} asks a dynamic-programming-specific question: once the physical MDP is fixed, how much endogenous memory must the temporal objective itself add to support repeated one-step continuation substitutions? In the bounded one-hole architecture, the answer is the ordinary temporal row rank even against all continuous Euclidean encoders. The Hankel object therefore belongs to the criterion's weight sequence rather than to the transition dynamics. Section~\ref{sec:discussion} makes the comparison with the closest realization results precise.

Sections~\ref{sec:setup}--\ref{sec:separation} develop the compatibility classes and sharp weighted theorem; Sections~\ref{sec:summability}--\ref{sec:abel} treat generalized summation, the Ramanujan obstruction, and the Abel benchmark; Section~\ref{sec:discussion} gives the literature comparison and scope limits. Supporting derivations are collected in the electronic companion.

\section{Setup and compatibility notions}
\label{sec:setup}

\subsection{Finite MDPs and expected stage-reward streams}

Let
\[
M=(S,A,P,r)
\]
be a finite MDP, using standard finite-state MDP notation \citep{puterman1994mdp}. The state space \(S\) is finite, each admissible action set \(A(s)\) is finite,
\[
P(s' \mid s,a)
\]
is the transition probability, and \(r(s,a)\in\mathbb{R}\) is a bounded stage reward. A policy class is always part of the problem specification; no result below silently transfers between stationary, nonstationary, randomized, and history-dependent policy classes.

For an initial state \(s\) and policy \(\pi\), define the expected stage reward at time \(t\) by
\[
g_t^\pi(s)
=
\mathbb{E}_s^\pi\!\left[r(S_t,A_t)\right],
\]
and the corresponding expected stage-reward stream by
\[
g^\pi(s)
=
\bigl(g_0^\pi(s),g_1^\pi(s),\ldots\bigr).
\]
An infinite-stream evaluation is a partial map
\[
E:D_E\to\mathbb{R},
\qquad
D_E\subseteq \mathbb{R}^{\Nzero}.
\]
Whenever \(g^\pi(s)\in D_E\), the policy value under \(E\) is
\[
J_E^\pi(s)=E(g^\pi(s)).
\]

This formulation intentionally separates the present setting from distributional objectives. The functional \(E\) acts after stagewise expectation has produced a deterministic stream. This expected-stream semantics also appears explicitly in recent finite-MDP valuation work \citep{neyman2023valuations,flesch2026finitelyadditive}.

\subsection{Stream operations}

For a stream \(x=(x_0,x_1,\ldots)\), define the left shift
\[
Sx=(x_1,x_2,\ldots).
\]
For \(c\in\mathbb{R}\), define the prefix operator
\[
c\oplus x=(c,x_0,x_1,\ldots),
\]
and write
\[
Jx=0\oplus x
\]
for zero prefixing. For a finite probability vector \(p=(p_1,\ldots,p_m)\) and streams \(x^1,\ldots,x^m\), define the coordinatewise mixture
\[
M_p(x^1,\ldots,x^m)
=
\sum_{i=1}^m p_i x^i.
\]

The strict-separation results are stated on
\[
X=[0,1]^{\Nzero}.
\]
For this stream class, let \(\BX\) denote the family of one-hole contexts
\begin{equation}
\label{eq:context}
B[x]
=
c\oplus
\left(
p\,x+\sum_{i=1}^m p_i y^i
\right),
\end{equation}
where \(m\ge 0\), \(c\in[0,1]\), the fixed streams \(y^i\) belong to \(X\), and
\[
p\ge 0,\qquad
p_i\ge 0,\qquad
p+\sum_{i=1}^m p_i=1.
\]
Repeated zero prefixing is included by taking \(c=0\), \(p=1\), and \(m=0\).

The family \(\BX\) abstracts the one-step structure ``current expected reward plus a finite mixture of continuations.'' In an arbitrary model class, the appropriate context family may be smaller; one should then restrict attention to contexts actually realizable by that model and policy class. The one-hole form is deliberate: it tests whether one continuation can be substituted while the other successor continuations are held fixed. It is therefore a necessary substitutability test for an exact multi-successor Bellman architecture on any model class that admits these one-hole substitutions, because such an architecture must remain valid after fixing all but one feasible successor continuation. We do not claim that closure or realizability for \(\BX\) alone is sufficient to construct a general multi-input Bellman backup.

\subsection{Why the full bounded stream cube is decision-relevant}

The full cube \(X\) is generated by an elementary finite MDP under a sufficiently rich policy class.

\begin{lemma}[Policy-stream richness]
\label{lem:richness}
Consider a one-state two-action self-loop MDP with actions \(a_0,a_1\) and rewards
\[
r(a_0)=0,
\qquad
r(a_1)=1.
\]
Under arbitrary nonstationary randomized policies, the set of expected stage-reward streams is exactly
\[
[0,1]^{\Nzero}.
\]
\end{lemma}

\proof{Proof.}
Let \(x=(x_0,x_1,\ldots)\in[0,1]^{\Nzero}\). At time \(t\), choose \(a_1\) with probability \(x_t\) and \(a_0\) with probability \(1-x_t\). Since the state self-loops and the stage rewards are \(0\) and \(1\),
\[
\mathbb{E}[r_t]=x_t.
\]
Thus every \(x\in X\) is policy-generated. Conversely, every expected reward lies in \([0,1]\), so every generated stream lies in \(X\).
\Halmos\endproof

\begin{keyremark}
The policy class is essential. Lemma~\ref{lem:richness} is false for substantially smaller classes such as stationary randomized policies, which generate only constant expected reward streams in this one-state model.
\end{keyremark}

\begin{keylemma}[Finite-MDP realization of the one-step contexts]
Every context \(B\in\BX\) in \eqref{eq:context} is realizable by a finite MDP under nonstationary randomized policies. More precisely, for fixed \(B\), the same finite MDP and the same one-step branching mechanism can be used for every variable continuation \(x\in X\); only the policy on the variable branch changes with \(x\).
\end{keylemma}

\proof{Proof.}
Fix
\[
B[x]
=
c\oplus\left(p\,x+\sum_{i=1}^m p_i y^i\right).
\]
Construct an initial state \(s_\star\) with a single action of reward \(c\). After that action, move to branch state \(s_0\) with probability \(p\) and to branch state \(s_i\) with probability \(p_i\), \(1\le i\le m\). At every branch state use two self-loop actions with rewards \(0\) and \(1\). On branch \(s_0\), at continuation time \(t\) choose the reward-one action with probability \(x_t\); on branch \(s_i\), choose it with probability \(y^i_t\). These are valid nonstationary randomized policies by Lemma~\ref{lem:richness}. The expected reward at the initial stage is \(c\), while at continuation time \(t\) it is
\[
p x_t+\sum_{i=1}^m p_i y^i_t.
\]
Hence the expected stage-reward stream is exactly \(B[x]\). The MDP and its first-step transition probabilities depend only on the fixed context \(B\), not on the variable continuation \(x\).
\Halmos\endproof

This lemma makes the role of \(\BX\) concrete: the strict-separation theorem is not driven by a larger algebraic context family that is unavailable to finite control models. Under the rich policy class of Lemma~\ref{lem:richness}, the selected prefix-and-mixture contexts are themselves finite-MDP realizable. The lemma does not turn the one-hole family into a sufficient Bellman architecture; rather, it shows that failure on this necessary subfamily is already decision-theoretically meaningful.

\subsection{Topologies on the stream cube}

We use two explicit topologies on \(X\).

First, let \(\prodtop\) be the product topology. Under \(\prodtop\), \(X\) is the Hilbert cube. Second, let \(\suptop\) be the topology induced by
\[
\|x-y\|_\infty
=
\sup_{t\ge 0}|x_t-y_t|.
\]
The phrase \emph{Hilbert cube} will be reserved for \((X,\prodtop)\).

For either topology, the finite-coordinate embedding
\[
\iota_m:[0,1]^m\to X,
\qquad
\iota_m(u_0,\ldots,u_{m-1})
=
(u_0,\ldots,u_{m-1},0,0,\ldots)
\]
is continuous.

\subsection{Compatibility classes}

All classes below are relative to the selected stream class \(X\) and context family \(\BX\).

\begin{keydefinition}[Domain composability]
Let \(E:D_E\to\mathbb{R}\). We write
\[
E\in\Dclass
\]
if \(X\subseteq D_E\) and
\[
B[x]\in D_E
\]
for every \(x\in X\) and every \(B\in\BX\).
\end{keydefinition}

\begin{keyremark}
The definition bundles two requirements: coverage of the selected stream class, \(X\subseteq D_E\), and closure under the selected constructors. For the particular family \(\BX\), every context maps \(X\) back into \(X\), so in the strict-separation setting of Section~\ref{sec:separation} constructor closure follows automatically from coverage. Accordingly, the strict inclusion \(\Fclass_\tau^{\aut}\subsetneq\Dclass\) there separates finite-dimensional realizability from total evaluability on the rich cube; it is not itself a nontrivial domain-closure pathology. Constructor closure becomes substantive for partial domains in Sections~\ref{sec:summability}--\ref{sec:abel}, with literal Ramanujan summation providing the explicit prefix-closure failure.
\end{keyremark}

For a general partial evaluation, domain coverage and constructor closure together form the first gate. If either fails, some policy-generated stream or selected one-step alternative is outside the evaluation domain, so the corresponding recursive representation is not yet well posed.

\begin{keydefinition}[Finite-dimensional continuous autonomous realizability]
Fix a topology \(\tau\) on \(X\). We write
\[
E\in\Fclass_\tau^{\aut}
\]
if \(E\in\Dclass\) and there exist a finite \(k\), a \(\tau\)-continuous encoder
\[
z:X\to\mathbb{R}^k,
\]
a continuous output map
\[
o:z(X)\to\mathbb{R},
\]
and, for every \(B\in\BX\), a continuous update
\[
T_B:z(X)\to z(X)
\]
such that
\[
z(B[x])=T_B(z(x))
\]
and
\[
E(x)=o(z(x))
\]
for every \(x\in X\).
\end{keydefinition}

The adjective \emph{autonomous} means that the continuation state is \(z(x)\), a function of the continuation stream itself. No calendar time or other exogenous coordinate is supplied to the state update. No continuity, measurability, or finite parameterization is imposed across the context index \(B\): each admissible context may have its own continuous update \(T_B\). Thus \(\Fclass_\tau^{\aut}\) is intentionally permissive as a necessary realization notion; the non-compression result below rules out finite-dimensional continuous encoding even with arbitrary context-specific continuous updates.

Continuity is a substantive regularity assumption. Without regularity, the identity $|X|=|\mathbb R|$ permits a purely set-theoretic injection of the entire stream cube into one real coordinate, and context updates can be transported through that encoding. Such encodings do not represent finite-memory compression in the topological sense considered here; the topology is therefore part of the realization model.

\begin{keydefinition}[Continuous scalar sufficiency]
We write
\[
E\in\Sclass_\tau^{\aut}
\]
if \(E\in\Fclass_\tau^{\aut}\) and the realization may be chosen with the scalar evaluation itself as state. Equivalently, \(E\) is \(\tau\)-continuous and for every \(B\in\BX\) there exists a continuous map
\[
W_B:E(X)\to\mathbb{R}
\]
such that
\[
E(B[x])=W_B(E(x))
\]
for every \(x\in X\).
\end{keydefinition}

\begin{keydefinition}[Contextual scalar order consistency]
We write
\[
E\in\Oclass_\tau^{\aut}
\]
if \(E\in\Sclass_\tau^{\aut}\) and each scalar context map \(W_B\) is nondecreasing. Equivalently,
\[
E(x)\le E(y)
\quad\Longrightarrow\quad
E(B[x])\le E(B[y])
\]
for all \(x,y\in X\) and \(B\in\BX\).
\end{keydefinition}

\begin{keydefinition}[Stagewise monotonicity]
Independently of the preceding chain, write
\[
E\in\Mclass
\]
if \(X\subseteq D_E\) and coordinatewise reward dominance on \(X\) implies value dominance:
\[
x,y\in X,\quad x_t\le y_t\ \ \forall t
\quad\Longrightarrow\quad
E(x)\le E(y).
\]
\end{keydefinition}

\begin{table}[t]
\TABLE
{Roadmap of the principal examples. ``Finite aut.'' denotes exact finite-dimensional continuous autonomous realization. Ramanujan and Abel use their own domains, so the table is descriptive rather than a common class-membership theorem.\label{tab:compatibility-roadmap}}
{\scriptsize
\begin{tabular}{@{}lccccc@{}}
\hline
Criterion & Domain behavior & Finite aut. & Scalar & Order & Monotone \\
\hline
Alternating geometric $E_{\pm}$ & total & pass & pass & fail & fail \\
Quasi-hyperbolic $E_{\beta,\delta}$ & total & pass & fail & -- & pass \\
Vanishing-ratio weighted $E_w$ & total & fail & -- & -- & pass \\
Negative geometric $E_-$ & total & pass & pass & pass & fail \\
Literal Candelpergher Ramanujan & prefix fail & -- & fail & -- & fail \\
Abel on $D_A$ & partial/closed & -- & alg.\ pass & alg.\ pass & pass \\
\hline
\end{tabular}}
{\scriptsize Notes: For Ramanujan, ``prefix fail'' is the controlled domain obstruction of Theorem~\ref{thm:ram-domain-obstruction}, and the scalar failure refers to Proposition~\ref{prop:ram-scalar-insufficiency} within $O^\pi_{\mathbb R}$. For Abel, ``partial/closed'' means that the finite-valued domain is proper but closed under the prefix and finite-mixture operations used in Proposition~\ref{prop:abel-compatibility}; ``alg.\ pass'' does not assert membership in the product-topology continuous classes.}
\end{table}

\paragraph{Architectural reading of the gates.}
The compatibility classes are definitions, not an additional characterization theorem. They should be read as nested necessary conditions for progressively more restrictive recursive-evaluation architectures. If all selected continuations must remain evaluable, the domain must cover the selected stream class and survive the selected constructors. If evaluation is propagated through a fixed finite-dimensional continuous autonomous state, the criterion must lie in \(\Fclass_\tau^{\aut}\); if that state is required to be the scalar \(E(x)\), it must lie in \(\Sclass_\tau^{\aut}\); and if common contexts must preserve the scalar continuation ranking, it must lie in \(\Oclass_\tau^{\aut}\). In particular, contextual order consistency is \emph{not} required for exact policy evaluation itself: \(\Oclass_\tau^{\aut}\) is the additional gate relevant when the scalar evaluation is also intended to carry a context-stable ranking of continuations.

These gates form \emph{Layer I}: structural well-posedness of recursive policy evaluation under the stated architecture. They are not sufficient for \emph{Layer II} optimal-control conclusions such as existence of optimal policies, stationary optimality, contraction, fixed-point uniqueness, or value-iteration convergence. Those questions require additional assumptions even after the relevant compatibility gate has been passed; recent infinite-stream MDP work independently shows that specifying a stream evaluation does not by itself settle optimal-strategy existence \citep{flesch2026finitelyadditive}.

By construction,
\[
\Oclass_\tau^{\aut}
\subseteq
\Sclass_\tau^{\aut}
\subseteq
\Fclass_\tau^{\aut}
\subseteq
\Dclass.
\]
The next section proves that every inclusion is strict and that stagewise monotonicity is independent of the structural chain in both directions needed below.

\section{Strict autonomous compatibility separations}
\label{sec:separation}

We next separate total evaluability on the selected rich cube, finite-dimensional autonomous realizability, scalar sufficiency, and contextual order consistency. On \(X\), the constructor-closure part of the first gate is automatic because every context in \(\BX\) maps \(X\) into itself. Accordingly, the final strict inclusion below is a coverage-versus-realizability separation; the genuinely nontrivial constructor-closure obstruction appears later in Theorem~\ref{thm:ram-domain-obstruction}.

\begin{theorem}[Strict autonomous compatibility separation]
\label{thm:strict-chain}
Let
\[
X=[0,1]^{\Nzero},
\]
and let \(\tau\in\{\prodtop,\suptop\}\). Then
\begin{equation}
\label{eq:strict-chain}
\Oclass_\tau^{\aut}
\subsetneq
\Sclass_\tau^{\aut}
\subsetneq
\Fclass_\tau^{\aut}
\subsetneq
\Dclass.
\end{equation}
Moreover,
\[
\Mclass\not\Rightarrow\Fclass_\tau^{\aut},
\qquad
\Oclass_\tau^{\aut}\not\Rightarrow\Mclass.
\]
\end{theorem}

The proof uses four explicit criteria. We first record a common continuity observation.

\begin{keylemma}[Absolutely summable weights]
Let \(w=(w_t)_{t\ge0}\in\ell^1\) and define
\[
E_w(x)=\sum_{t=0}^\infty w_t x_t,
\qquad
x\in X.
\]
Then \(E_w\) is continuous under both \(\prodtop\) and \(\suptop\).
\end{keylemma}

\proof{Proof.}
For the sup-norm topology,
\[
|E_w(x)-E_w(y)|
\le
\|w\|_1\|x-y\|_\infty.
\]
For the product topology, the finite-coordinate partial sums
\[
E_w^{(N)}(x)=\sum_{t=0}^N w_t x_t
\]
are continuous, while
\[
\sup_{x\in X}|E_w(x)-E_w^{(N)}(x)|
\le
\sum_{t>N}|w_t|
\longrightarrow 0.
\]
Thus \(E_w\) is a uniform limit of continuous functions.
\Halmos\endproof

\subsection{Scalar sufficiency need not preserve continuation order}

Fix \(0<\gamma<1\) and define
\[
E_{\pm}(x)
=
\sum_{t=0}^{\infty}(-\gamma)^t x_t.
\]

\begin{keyproposition}
For either \(\tau\in\{\prodtop,\suptop\}\),
\[
E_{\pm}
\in
\Sclass_\tau^{\aut}
\setminus
\Oclass_\tau^{\aut}.
\]
\end{keyproposition}

\proof{Proof.}
Continuity follows from the absolute-summability lemma above. For a context
\[
B[x]
=
c\oplus
\left(
p\,x+\sum_i p_i y^i
\right),
\]
linearity gives
\[
E_{\pm}(B[x])
=
c-\gamma p\,E_{\pm}(x)
-\gamma\sum_i p_iE_{\pm}(y^i).
\]
For fixed \(B\), this is a continuous function of the single scalar \(E_{\pm}(x)\), so scalar sufficiency holds.

Contextual order consistency fails whenever \(p>0\), because the coefficient of \(E_{\pm}(x)\) is negative. For example, let \(x=0\), \(y=(1,0,0,\ldots)\), and take the zero-prefix context \(J\). Then
\[
E_{\pm}(x)=0<1=E_{\pm}(y),
\]
but
\[
E_{\pm}(Jx)=0>-\gamma=E_{\pm}(Jy).
\]
\Halmos\endproof

This establishes
\[
\Oclass_\tau^{\aut}
\subsetneq
\Sclass_\tau^{\aut}.
\]

\subsection{Finite-dimensional sufficiency need not reduce to one scalar}

Fix \(0<\beta<1\) and \(0<\delta<1\), and define the quasi-hyperbolic criterion
\[
E_{\beta,\delta}(x)
=
x_0
+
\beta\sum_{t=1}^{\infty}\delta^t x_t.
\]
Also define the geometric continuation statistic
\[
U_\delta(x)
=
\sum_{t=0}^{\infty}\delta^t x_t.
\]

\begin{keyproposition}
For either \(\tau\in\{\prodtop,\suptop\}\),
\[
E_{\beta,\delta}
\in
\Fclass_\tau^{\aut}
\setminus
\Sclass_\tau^{\aut}.
\]
\end{keyproposition}

\proof{Proof.}
Both \(E_{\beta,\delta}\) and \(U_\delta\) are continuous by the absolute-summability lemma above. Consider the two-dimensional state
\[
z(x)
=
\begin{pmatrix}
E_{\beta,\delta}(x)\\
U_\delta(x)
\end{pmatrix}.
\]
For
\[
B[x]
=
c\oplus
\left(
p\,x+\sum_i p_i y^i
\right),
\]
let
\[
C_B
=
\sum_i p_iU_\delta(y^i).
\]
Then
\[
U_\delta(B[x])
=
c+\delta\bigl(p\,U_\delta(x)+C_B\bigr),
\]
and
\[
E_{\beta,\delta}(B[x])
=
c+\beta\delta\bigl(p\,U_\delta(x)+C_B\bigr).
\]
Hence \(z(B[x])\) is a continuous affine function of \(z(x)\); in fact, only the second state coordinate is needed for the update. This gives an exact two-dimensional continuous autonomous realization.

To show that the scalar \(E_{\beta,\delta}\) is not sufficient, take
\[
x=(\beta\delta,0,0,\ldots),
\qquad
y=(0,1,0,\ldots).
\]
Both belong to \(X\), and
\[
E_{\beta,\delta}(x)
=
E_{\beta,\delta}(y)
=
\beta\delta.
\]
After zero prefixing,
\[
E_{\beta,\delta}(Jx)
=
\beta^2\delta^2,
\qquad
E_{\beta,\delta}(Jy)
=
\beta\delta^2.
\]
These values differ because \(0<\beta<1\). Thus no scalar update based only on \(E_{\beta,\delta}(x)\) can implement the zero-prefix context.
\Halmos\endproof

Consequently,
\[
\Sclass_\tau^{\aut}
\subsetneq
\Fclass_\tau^{\aut}.
\]

\subsection{Temporal Hankel rank and sharp autonomous memory dimension}

We now identify the exact autonomous state dimension for weighted criteria. Let
\[
w=(w_n)_{n\ge0}\in\ell^1
\]
be a real weight sequence and define
\begin{equation}
\label{eq:weighted-criterion}
E_w(x)
=
\sum_{n=0}^{\infty}w_nx_n.
\end{equation}
Absolute summability makes $E_w$ finite on all of $X$ and continuous under both selected topologies. We use the coordinate left shift from Section~\ref{sec:setup} also on weights,
\[
S(w_0,w_1,w_2,\ldots)=(w_1,w_2,\ldots),
\]
and define the temporal shift space and its algebraic dimension by
\begin{equation}
\label{eq:temporal-shift-space}
\mathcal W_w
:=
\operatorname{span}\{S^iw:i\ge0\},
\qquad
r_w:=\dim\mathcal W_w\in\Nzero\cup\{\infty\}.
\end{equation}
Also define the infinite temporal Hankel matrix
\begin{equation}
\label{eq:temporal-hankel}
H_w=(w_{i+j})_{i,j\ge0}.
\end{equation}
Here $\operatorname{rank}_{\mathrm{alg}}H_w$ denotes algebraic row rank, equivalently the supremum of the ranks of its finite submatrices.

\begin{lemma}[Temporal shift rank equals temporal Hankel rank]
\label{lem:weighted-hankel-rank}
For every $w\in\ell^1$,
\[
r_w
=
\operatorname{rank}_{\mathrm{alg}}H_w.
\]
\end{lemma}

\proof{Proof.}
The $i$th row of $H_w$ is exactly
\[
(w_i,w_{i+1},\ldots)=S^iw.
\]
Hence the algebraic row span of $H_w$ is $\mathcal W_w$.
\Halmos\endproof

The algebraic backbone is classical. In one variable, finite Hankel rank is equivalent to a finite constant-coefficient linear recurrence and to rationality of the corresponding generating series; formal-series realization theory also identifies Hankel rank with minimal linear-representation dimension \citep{schutzenberger1961automata,carlylepaz1971realizations,fliess1974hankel}. Nonlinear and topological realization theory also substantially predates the present setting: Sussmann~\cite{sussmann1976minimal} establishes minimal realization results for analytic-manifold nonlinear systems, Ehrig and K\"uhnel~\cite{ehrig1974topological} develop reduction and minimization for deterministic topological automata with continuous state transitions, and Sontag~\cite{sontag1979realization} treats bounded discrete-time nonlinear responses and state-affine/state-linear realizations. Current structured realization theory continues this line \citep{petreczky2026stability}. Rank-type lower bounds also occur in structured nonlinear realization theory: Petreczky and Vidal~\cite{petreczky2008semialgebraic} define an \emph{algebraic Hankel rank} through Krull dimensions of algebras generated by shifted output time series and obtain state-dimension lower bounds for semi-algebraic and implicit-polynomial systems. These results use different invariants, regularity classes, or notions of minimality from the ordinary temporal row rank and Euclidean continuity-only competitor class considered below. The theorem is therefore stated only for the bounded continuation-stream architecture defined in Section~\ref{sec:setup}; it is not a general nonlinear-realization claim.

\begin{theorem}[Sharp temporal-Hankel characterization of continuous autonomous memory]
\label{thm:weighted-hankel-characterization}
Let $w\in\ell^1$ and let $E_w$ be given by \eqref{eq:weighted-criterion}. For either $\tau\in\{\prodtop,\suptop\}$,
\[
E_w\in\Fclass_\tau^{\aut}
\quad\Longleftrightarrow\quad
r_w<\infty.
\]
More precisely, let $k_{\min,\tau}^{\mathrm{cont,aut}}(E_w)$ be the minimum Euclidean encoder dimension among exact autonomous realizations satisfying the continuity requirements of $\Fclass_\tau^{\aut}$, with value $\infty$ if no finite-dimensional realization exists. Then
\begin{equation}
\label{eq:sharp-hankel-dimension}
k_{\min,\tau}^{\mathrm{cont,aut}}(E_w)
=
r_w
=
\operatorname{rank}_{\mathrm{alg}}H_w.
\end{equation}
For $w=0$, the equality uses the zero-dimensional singleton state $\mathbb R^0$.
\end{theorem}

\proof{Proof.}
We first prove the upper bound. Suppose $r_w=m<\infty$ and choose a basis
\[
a^1,\ldots,a^m
\]
of $\mathcal W_w$. Define
\[
z:X\to\mathbb R^m,
\qquad
z_r(x)=\langle a^r,x\rangle.
\]
Each $a^r$ belongs to $\ell^1$, so every coordinate of $z$ is continuous under both $\prodtop$ and $\suptop$. Because $S\mathcal W_w\subseteq\mathcal W_w$, there is a matrix $A\in\mathbb R^{m\times m}$ such that
\[
Sa^r
=
\sum_{s=1}^m A_{rs}a^s,
\qquad r=1,\ldots,m.
\]
Fix a context
\[
B[x]
=
c\oplus\left(p\,x+\sum_i p_i y^i\right)
\in\BX.
\]
Then
\[
\begin{aligned}
z_r(B[x])
&=
a_0^r c
+
\left\langle Sa^r,
p\,x+\sum_i p_i y^i
\right\rangle\\
&=
a_0^r c
+
p\sum_{s=1}^m A_{rs}z_s(x)
+
\sum_i p_i\langle Sa^r,y^i\rangle.
\end{aligned}
\]
Thus
\[
z(B[x])=pAz(x)+b_B
\]
for a vector $b_B$ depending only on the fixed context. Restricting the affine map $u\mapsto pAu+b_B$ to $z(X)$ gives the required continuous state update. Since $w\in\mathcal W_w$, write $w=\sum_{r=1}^m\alpha_r a^r$; then
\[
E_w(x)=\alpha^\top z(x),
\]
so the output is linear. Hence $k_{\min,\tau}^{\mathrm{cont,aut}}(E_w)\le m$.

For the lower bound, suppose an exact autonomous realization has a $\tau$-continuous encoder
\[
z:X\to\mathbb R^k.
\]
Assume $r_w>k$. Choose $k+1$ linearly independent shifted weights
\[
a^r=S^{i_r}w,
\qquad
r=1,\ldots,k+1.
\]
Their infinite row matrix has row rank $k+1$, so there exist coordinate indices $n_1,\ldots,n_{k+1}$ such that
\[
M=(a^r_{n_s})_{r,s=1}^{k+1}
\]
is nonsingular. Let $Q=M^{-1}$ and define finitely supported sequences $d_1,\ldots,d_{k+1}$ by
\[
(d_s)_{n_\ell}=Q_{\ell s},
\qquad
(d_s)_n=0
\quad\text{for }n\notin\{n_1,\ldots,n_{k+1}\}.
\]
Then
\begin{equation}
\label{eq:biorthogonal-perturbations}
\langle a^r,d_s\rangle=\delta_{rs}.
\end{equation}
Set $D=\sum_{s=1}^{k+1}\|d_s\|_\infty$ and choose $0<\varepsilon<1/(2D)$. For
\[
t=(t_1,\ldots,t_{k+1})\in[-1,1]^{k+1},
\]
define
\begin{equation}
\label{eq:residual-cube-embedding}
x(t)
=
\frac12\mathbf 1
+
\varepsilon\sum_{s=1}^{k+1}t_sd_s.
\end{equation}
The choice of $\varepsilon$ gives $x(t)\in X$. Because all perturbation directions have finite support, $t\mapsto x(t)$ is continuous under both selected topologies. Moreover,
\begin{equation}
\label{eq:residual-cube-coordinates}
E_w(J^{i_r}x(t))
=
\langle S^{i_r}w,x(t)\rangle
=
\frac12\langle a^r,\mathbf1\rangle
+
\varepsilon t_r.
\end{equation}
If $z(x(t))=z(x(t'))$, exact autonomy for the zero-prefix context implies
\[
z(J^ix(t))=z(J^ix(t'))
\qquad\forall i\ge0,
\]
and exact output recovery therefore gives
\[
E_w(J^ix(t))=E_w(J^ix(t'))
\qquad\forall i\ge0.
\]
Applying \eqref{eq:residual-cube-coordinates} at $i=i_r$ yields $t_r=t'_r$ for every $r$. Hence
\[
t\longmapsto z(x(t))
\]
is a continuous injection from $[-1,1]^{k+1}$ into $\mathbb R^k$. The case $k=0$ is already impossible. For $k\ge1$, restrict to the open cube $(-1,1)^{k+1}$ and compose with the standard inclusion $\mathbb R^k\hookrightarrow\mathbb R^{k+1}$. Brouwer's invariance-of-domain theorem \citep[see, e.g.,][]{munkres2000topology} would make the image open in $\mathbb R^{k+1}$, but that image lies in a $k$-dimensional hyperplane with empty interior. This contradiction proves $r_w\le k$.

Combining the lower bound with the affine upper construction gives \eqref{eq:sharp-hankel-dimension}; if $r_w=\infty$, no finite $k$ is possible.
\Halmos\endproof

The lower bound uses continuity only for the encoder $z$; linearity, affinity, differentiability, and analyticity of a competing realization are not assumed. The stronger continuity requirements in the definition of $\Fclass_\tau^{\aut}$ are retained because they define the comparison class used throughout the paper.

\begin{corollary}[Criterion-induced memory already on a one-state MDP]
\label{cor:criterion-memory-cost}
Fix $w\in\ell^1$ and either selected topology $\tau$. In the one-state two-action self-loop MDP of Lemma~\ref{lem:richness}, under arbitrary nonstationary randomized policies, any exact $\tau$-continuous autonomous continuation encoder that supports repeated zero-reward one-step prefixing and exact recovery of $E_w$ has Euclidean dimension at least $r_w$. If $r_w<\infty$, the $r_w$-dimensional affine encoder constructed in Theorem~\ref{thm:weighted-hankel-characterization} supports the entire context family $\BX$. Hence $r_w$ is the exact endogenous continuous memory dimension of the weighted criterion for the selected clock-free continuation architecture, even when the physical MDP has only one state.
\end{corollary}

\proof{Proof.}
The lower-bound proof above uses only $J$ and its iterates, all realizable in the one-state model by taking the zero-reward action for the required prefixes before following the policy generating $x$. The affine construction supplies the matching upper bound for the full context family.
\Halmos\endproof

This corollary separates criterion memory from environment memory. The physical Markov state is fixed and trivial; the extra coordinates are required solely to propagate the objective's temporal weighting without an externally supplied clock. Geometric discounting costs one endogenous coordinate, quasi-hyperbolic weighting costs two, a $K$-mode exponential mixture costs $K$ after repeated modes are combined, and infinite temporal rank requires infinitely many coordinates in the exact continuous autonomous class.

\begin{keycorollary}[Positive vanishing-ratio criteria have infinite temporal rank]
Suppose
\begin{equation}
\label{eq:vanishing-ratio}
w_n>0
\quad\text{for all }n,
\qquad
\frac{w_{n+1}}{w_n}\longrightarrow0.
\end{equation}
Then $w\in\ell^1$, the shifts $\{S^iw:i\ge0\}$ are linearly independent, and therefore
\[
r_w=\infty,
\qquad
E_w\notin\Fclass_\tau^{\aut}
\]
for either selected topology.
\end{keycorollary}

\proof{Proof.}
Because $w_{n+1}/w_n\to0$, the ratio is eventually at most $1/2$, so $w\in\ell^1$. If a nontrivial finite relation
\[
\sum_{j=0}^m c_jS^jw=0
\]
held, let $j_*$ be the smallest index with $c_{j_*}\neq0$. Looking at coordinate $n$ and dividing by the positive number $w_{n+j_*}$ gives
\[
c_{j_*}
+
\sum_{j>j_*}c_j\frac{w_{n+j}}{w_{n+j_*}}
=0.
\]
For every fixed $j>j_*$, the ratio on the right is a finite product of consecutive ratios $w_{q+1}/w_q$ and tends to zero. Letting $n\to\infty$ yields $c_{j_*}=0$, a contradiction. Hence all finite subsets of the shift orbit are linearly independent.
\Halmos\endproof

For this positive vanishing-ratio subclass, $E_w$ is total and continuous under both selected topologies, positivity gives stagewise monotonicity, and $\BX$ maps $X$ into itself. Hence
\[
E_w\in\Dclass\cap\Mclass
\qquad\text{but}\qquad
E_w\notin\Fclass_\tau^{\aut}.
\]
Thus divergence, partial-domain failure, and sign reversal have all been removed before the representation obstruction appears.

Concrete examples are immediate. If $0<q<1$ and $p>1$, then
\[
w_n=q^{n^p}
\quad\Longrightarrow\quad
\frac{w_{n+1}}{w_n}
=
q^{(n+1)^p-n^p}
\longrightarrow0,
\]
so every criterion
\[
E_{q,p}(x)=\sum_{n\ge0}q^{n^p}x_n,
\qquad p>1,
\]
has infinite temporal Hankel rank; the former $q^{n^2}$ witness is the case $p=2$. Likewise, $w_n=(n!)^{-\alpha}$ with $\alpha>0$ has infinite rank because $w_{n+1}/w_n=(n+1)^{-\alpha}\to0$. The ratio-zero condition is only a convenient sufficient condition, not the theorem boundary.

The sharp characterization also covers regimes omitted by that sufficient condition. For example, if
\[
w_n=(n+1)^{-p},
\qquad p>1,
\]
then $w\in\ell^1$ but $r_w=\infty$. Indeed, finite rank would imply a finite constant-coefficient recurrence and hence, after combining equal characteristic roots, a representation
\[
w_n=\sum_{j=1}^m P_j(n)\lambda_j^n
\]
for all sufficiently large $n$, with polynomials $P_j$ and active roots $\lambda_j$ \citep{fliess1974hankel}. Because $w\in\ell^1$, the generating series $W(z)=\sum_{n\ge0}w_nz^n$ converges absolutely on $|z|\le1$. In the rational generating function induced by the recurrence, an active root $\lambda_j$ corresponds to a pole at $z=\lambda_j^{-1}$; absolute convergence on the closed unit disk therefore forces $|\lambda_j|<1$ for every active root. Hence, for some $\rho<1$, integer $d\ge0$, and constant $C>0$, one has $|w_n|\le Cn^d\rho^n$ eventually. This is impossible for $w_n=(n+1)^{-p}$, since $(n+1)^{-p}/(n^d\rho^n)\to\infty$. Therefore $r_w=\infty$. The usual nonsummable hyperbolic case $p=1$ remains outside the total bounded-stream setting used here.

At the finite-rank end, geometric and alternating-geometric weights have $r_w=1$. For the quasi-hyperbolic weights from the preceding subsection,
\[
w=(1,\beta\delta,\beta\delta^2,\ldots),
\qquad 0<\beta,\delta<1,
\]
we have $S^2w=\delta Sw$ while $w$ and $Sw$ are linearly independent because $\beta\neq1$, so $r_w=2$. Consequently, the two-dimensional realization constructed above is dimension-minimal even among arbitrary continuous nonlinear autonomous encoders. More generally, after combining repeated modes, a finite mixture
\[
w_n=\sum_{r=1}^K a_r\beta_r^n
\]
with distinct $\beta_r$ and nonzero $a_r$ has temporal rank $K$, and polynomial-times-exponential weights have finite rank determined by the usual recurrence multiplicities. These are classical finite-mode cases \citep{feinberg1994weighted,fliess1974hankel}.

\begin{remark}[Scope of the sharp weighted theorem]
\label{rem:weighted-scope}
Theorem~\ref{thm:weighted-hankel-characterization} is an exact, autonomous representation result. It does not exclude discontinuous encodings, infinite-dimensional states, approximate finite-dimensional representations, explicit calendar-time augmentation, or nonstationary dynamic programming. For any fixed weight sequence $w$, a re-indexed continuation $x=(x_0,x_1,\ldots)$ available at calendar time $t$ has the scalar nonautonomous representation
\[
V_t^w(x)
=
\sum_{j=0}^{\infty}w_{t+j}x_j,
\]
with
\[
V_t^w(c\oplus x)
=
w_t c+V_{t+1}^w(x).
\]
Thus infinite temporal rank rules out compression into a fixed finite-dimensional $\tau$-continuous \emph{autonomous} continuation state, not the use of an externally supplied clock. This is consistent with the general-discounting literature \citep{carmon2009exponential}.
\end{remark}

The positive vanishing-ratio subclass, together with Theorem~\ref{thm:weighted-hankel-characterization}, establishes
\[
\Fclass_\tau^{\aut}
\subsetneq
\Dclass
\]
and, because every criterion in that subclass lies in $\Mclass$, also shows
\[
\Mclass\not\Rightarrow\Fclass_\tau^{\aut}.
\]

\subsection{Contextual order consistency is distinct from stagewise monotonicity}

Finally, fix \(0<\gamma<1\) and define
\[
E_-(x)
=
-\sum_{t=0}^{\infty}\gamma^t x_t.
\]

\begin{keyproposition}
For either \(\tau\in\{\prodtop,\suptop\}\),
\[
E_-\in\Oclass_\tau^{\aut}
\qquad\text{but}\qquad
E_-\notin\Mclass.
\]
\end{keyproposition}

\proof{Proof.}
Continuity follows from the absolute-summability lemma above. For a context
\[
B[x]
=
c\oplus
\left(
p\,x+\sum_i p_i y^i
\right),
\]
we have
\[
E_-(B[x])
=
-c+\gamma p\,E_-(x)
+
\gamma\sum_i p_iE_-(y^i).
\]
The coefficient \(\gamma p\) is nonnegative, so every scalar context update is nondecreasing and \(E_-\in\Oclass_\tau^{\aut}\).

Stagewise monotonicity fails. For example, \(0\le(1,0,0,\ldots)\) coordinatewise, but
\[
E_-(0)=0>-1=E_-(1,0,0,\ldots).
\]
\Halmos\endproof

\proof{Proof of Theorem~\ref{thm:strict-chain}.}
The inclusions
\[
\Oclass_\tau^{\aut}
\subseteq
\Sclass_\tau^{\aut}
\subseteq
\Fclass_\tau^{\aut}
\subseteq
\Dclass
\]
hold by definition. The alternating-geometric construction makes the first inclusion strict, the quasi-hyperbolic construction makes the second strict, and Theorem~\ref{thm:weighted-hankel-characterization}, witnessed by the positive vanishing-ratio subclass, makes the third strict. The same subclass gives
\(\Mclass\not\Rightarrow\Fclass_\tau^{\aut}\). The negative-geometric construction gives
\(\Oclass_\tau^{\aut}\not\Rightarrow\Mclass\).
\Halmos\endproof

\subsection{Interpretation}

Theorem~\ref{thm:strict-chain} identifies distinct failure modes within the total-cube architecture of Section~\ref{sec:setup}. Its final strict inclusion is specifically a total-coverage-versus-realizability separation, not a second proof of partial-domain constructor failure. The role of the vanishing-ratio weighted class is to make the representation layer nontrivial even after the simpler explanations have been removed: the criterion is total, continuous, and stagewise monotone, yet exact autonomous recursion still requires more information than any fixed finite-dimensional continuous Euclidean state can retain. By contrast, the quasi-hyperbolic witness shows a later scalar-sufficiency failure, and the alternating-geometric witness shows that exact scalar recursion can still fail to preserve contextual order. Stagewise reward monotonicity is independent of these structural properties.

These are structural, not algorithmic, statements. They neither establish nor preclude optimal-policy existence, stationary optimality, contraction, Bellman fixed-point uniqueness, or value-iteration convergence. The next sections apply the same distinction to generalized summation.

\section{Generalized summation and scalar Bellman compatibility}
\label{sec:summability}

The preceding separation results were stated on the bounded stream cube used to model a rich finite-MDP policy class. We now change mathematical settings. This section concerns generalized summation on a domain
\[
D\subseteq\mathbb{R}^{\Nzero},
\]
which need not be bounded and should not be identified with the stream space of Section~\ref{sec:setup}. The purpose is to isolate what scalar one-step recursion requires from a summation rule before any particular generalized method is considered.

For any partial evaluation \(E:D\to\mathbb{R}\), we say that \(E\) is \emph{stagewise monotone on \(D\)} if, for all \(x,y\in D\), coordinatewise dominance \(x_t\le y_t\) for every \(t\) implies \(E(x)\le E(y)\). This extends the notation \(\Mclass\) from Section~\ref{sec:setup}, where the domain is the full cube \(X\).

Let
\[
c_{00}
=
\left\{
x\in\mathbb{R}^{\Nzero}:
 x_t=0\text{ for all sufficiently large }t
\right\}
\]
denote the finite-support real streams, and assume throughout this subsection that
\[
c_{00}\subseteq D.
\]

\begin{keydefinition}[Finite-support regularity]
An evaluation
\[
E:D\to\mathbb{R}
\]
is \emph{finitely regular} if
\[
E(x)=\sum_{t=0}^{\infty}x_t
\qquad
\text{for every }x\in c_{00}.
\]
The sum on the right is finite because each such stream has finite support.
\end{keydefinition}

Finite-support regularity is deliberately weaker than requiring agreement with every ordinarily convergent infinite series. It is nevertheless enough to determine the only possible scalar prefix update.

\begin{keyremark}[Criteria outside the summation setting]
Proposition~\ref{thm:translation-necessity} is a statement about generalized summation rules satisfying finite-support regularity, not a universal restriction on all undiscounted infinite-horizon criteria. For example, whenever the long-run average
\[
E_{\mathrm{avg}}(x)
=
\lim_{N\to\infty}\frac{1}{N+1}\sum_{t=0}^{N}x_t
\]
exists, every finite-support stream has long-run average zero. Hence this criterion is not finitely regular on $c_{00}$ and lies outside the assumptions of Proposition~\ref{thm:translation-necessity}.
\end{keyremark}

\begin{keydefinition}[Scalar-prefix compatibility]
Assume that \(D\) is closed under arbitrary real prefixing:
\[
x\in D,\ c\in\mathbb{R}
\quad\Longrightarrow\quad
c\oplus x\in D.
\]
We say that \(E\) is \emph{scalar-prefix compatible} if there exists a single map
\[
W:\mathbb{R}^2\to\mathbb{R}
\]
such that
\[
E(c\oplus x)=W(c,E(x))
\]
for every \(c\in\mathbb{R}\) and \(x\in D\).
\end{keydefinition}

The use of one common map \(W\) is essential: the current term and the scalar continuation value must determine the prefixed value without access to any further information about the continuation stream.

\begin{proposition}[Scalar-prefix translation necessity]
\label{thm:translation-necessity}
Suppose that:
\begin{enumerate}
\item[(i)] \(c_{00}\subseteq D\);
\item[(ii)] \(D\) is closed under arbitrary real prefixing;
\item[(iii)] \(E:D\to\mathbb{R}\) is finitely regular; and
\item[(iv)] \(E\) is scalar-prefix compatible through one common map \(W\).
\end{enumerate}
Then
\begin{equation}
\label{eq:translation-W}
W(c,v)=c+v
\qquad
\text{for all }c,v\in\mathbb{R}.
\end{equation}
Consequently,
\begin{equation}
\label{eq:translation-prefix}
E(c\oplus x)=c+E(x)
\qquad
\text{for all }c\in\mathbb{R},\ x\in D.
\end{equation}
No linearity assumption on \(E\) is required.
\end{proposition}

\proof{Proof.}
Fix arbitrary \(v\in\mathbb{R}\) and consider the finite-support stream
\[
x^{(v)}=(v,0,0,\ldots).
\]
Finite-support regularity gives
\[
E(x^{(v)})=v.
\]
For arbitrary \(c\in\mathbb{R}\), the stream
\[
c\oplus x^{(v)}=(c,v,0,0,\ldots)
\]
also has finite support, so
\[
E(c\oplus x^{(v)})=c+v.
\]
Scalar-prefix compatibility gives, for the same stream,
\[
E(c\oplus x^{(v)})
=
W(c,E(x^{(v)}))
=
W(c,v).
\]
Hence \(W(c,v)=c+v\). Since \(c\) and \(v\) were arbitrary, \eqref{eq:translation-W} holds on all of \(\mathbb{R}^2\). Substituting \(v=E(x)\) yields \eqref{eq:translation-prefix} for every \(x\in D\).
\Halmos\endproof

If \(D\) is also shift closed, then every \(x\in D\) satisfies \(x=x_0\oplus Sx\), and Proposition~\ref{thm:translation-necessity} immediately gives
\begin{equation}
\label{eq:translativity}
E(x)=x_0+E(Sx).
\end{equation}
Equation~\eqref{eq:translativity} is the classical \emph{translativity} or stability-under-translation law of summability theory, not a new summability axiom introduced here; standard accounts treat adding or deleting finitely many terms with the corresponding ordinary adjustment as a classical property of summation methods \citep{cooke1950infinite,hardy1949divergent}. Nor is translation invariance a new consequence of recursivity in decision theory: monotone recursive-preference results already impose translation-invariant certainty equivalents in substantially richer stochastic preference settings \citep{bommier2017monotone,marinacci2026recursive}. Proposition~\ref{thm:translation-necessity} is narrower and assumption-specific: for a generalized summation rule, finite-support regularity plus a common scalar prefix aggregator already force the additive prefix law, without linearity or monotonicity.

\begin{corollary}[Linear translativity and scalar Bellman policy evaluation]
\label{cor:linear-bellman}
Assume, in addition, that \(D\) is shift closed, is a vector space containing the relevant finite coordinatewise mixtures, and that \(E\) is linear. Then finite-support regularity together with scalar-prefix compatibility is equivalent to the translative prefix law
\[
E(c\oplus x)=c+E(x).
\]
In either formulation, for every finite probability vector \((p_i)\) and streams \(x^i\in D\) for which the mixture belongs to \(D\),
\begin{equation}
\label{eq:linear-bellman-law}
E\left(
c\oplus\sum_i p_i x^i
\right)
=
c+\sum_i p_iE(x^i).
\end{equation}
Thus, in this linear mixture-closed setting, classical translativity is precisely the scalar one-step Bellman policy-evaluation law.
\end{corollary}

\proof{Proof.}
Proposition~\ref{thm:translation-necessity} gives the translative prefix law from finite-support regularity and scalar-prefix compatibility.

Conversely, assume the translative prefix law. It immediately gives scalar-prefix compatibility by taking
\[
W(c,v)=c+v.
\]
It also recovers finite-support regularity from linearity. Indeed, linearity gives \(E(0)=0\). If
\[
x=(x_0,\ldots,x_N,0,0,\ldots)
\]
has finite support, repeated use of translativity yields
\[
E(x)
=
x_0+E(Sx)
=
\cdots
=
\sum_{t=0}^{N}x_t+E(0)
=
\sum_{t=0}^{N}x_t.
\]
Thus the two formulations are equivalent under the stated linearity and domain assumptions.

Finally, linearity gives
\[
E\left(\sum_i p_i x^i\right)
=
\sum_i p_iE(x^i),
\]
and applying the prefix law yields \eqref{eq:linear-bellman-law}.
\Halmos\endproof

The corollary is intentionally a policy-evaluation statement. It does not by itself imply optimal-policy existence, stationary optimality, contraction of an optimality operator, uniqueness of a Bellman fixed point, or convergence of value iteration.

\subsection{Monotonicity excludes persistent positive divergence}

The previous result uses scalar recursive sufficiency. A different obstruction follows from finite-support regularity and stagewise monotonicity alone.

For \(x=(x_0,x_1,\ldots)\), define its \(N\)-th finite truncation by
\[
x^{(N)}
=
(x_0,\ldots,x_N,0,0,\ldots).
\]

\begin{corollary}[Nonnegative positive-divergence exclusion]
\label{cor:positive-divergence}
Let \(E:D\to\mathbb{R}\) be finitely regular and stagewise monotone. Then \(D\) contains no stream \(x\) such that
\[
x_t\ge0
\qquad
\text{for all }t,
\]
every truncation \(x^{(N)}\) belongs to \(D\), and
\[
\sum_{t=0}^{N}x_t\longrightarrow+\infty.
\]
Equivalently, any extension that includes such an \(x\) while preserving finite-support regularity and stagewise monotonicity cannot remain real valued at \(x\).
\end{corollary}

\proof{Proof.}
Suppose for contradiction that such an \(x\) belongs to \(D\). For every \(N\),
\[
x^{(N)}\le x
\]
coordinatewise. Stagewise monotonicity therefore gives
\[
E(x^{(N)})\le E(x).
\]
By finite-support regularity,
\[
E(x^{(N)})
=
\sum_{t=0}^{N}x_t.
\]
Because \(E:D\to\mathbb{R}\) is finite valued, membership \(x\in D\) would make \(E(x)\) a finite common upper bound for all partial sums, contradicting their divergence to \(+\infty\). Hence no such \(x\) can lie in \(D\).
\Halmos\endproof

The conclusion is deliberately one-sided. It does not say that every stream whose divergence is caused by cancellation or oscillation admits a useful generalized value. It says only that a finite-valued finitely regular monotone extension cannot assign a finite value to a nonnegative stream with unbounded ordinary partial sums.

\begin{keyproposition}[Persistent constant streams]
Suppose \(E\) satisfies the hypotheses of Proposition~\ref{thm:translation-necessity}. Then no nonzero constant stream
\[
\mathbf c=(c,c,c,\ldots),
\qquad c\ne0,
\]
can belong to \(D\).
\end{keyproposition}

\proof{Proof.}
If $\mathbf c\in D$, then $\mathbf c=c\oplus\mathbf c$ and Proposition~\ref{thm:translation-necessity} would imply
\[
E(\mathbf c)=c+E(\mathbf c),
\]
which is impossible in $\mathbb R$ when $c\ne0$.
\Halmos\endproof

This obstruction does not require stagewise monotonicity. A generalized summation that assigns a finite value to a nonzero persistent constant stream must therefore give up at least one hypothesis of Proposition~\ref{thm:translation-necessity}; Section~\ref{sec:ramanujan} will show that literal Candelpergher Ramanujan summation does so through its nonclassical translation behavior and, more fundamentally for control, through its domain structure.

\subsection{A finite-state Poisson gate}

The translation law has an immediate consequence for finite Markov reward processes. Let \(P\in\mathbb{R}^{m\times m}\) be a finite stochastic matrix and let \(r\in\mathbb{R}^m\) be a stage-reward vector. For state \(s\), define the expected stage-reward stream
\begin{equation}
\label{eq:mrp-stream}
g_s
=
\left(
e_s^\top r,
e_s^\top Pr,
e_s^\top P^2r,
\ldots
\right).
\end{equation}
Let \(\Pi\) denote the Ces\`aro, equivalently spectral, projection of \(P\) onto its eigenvalue-\(1\) subspace.

\begin{corollary}[Finite-state zero-gain necessity]
\label{cor:zero-gain}
Suppose \(E\) is a linear translative evaluation on a vector-space domain containing every stream \(g_s\) in \eqref{eq:mrp-stream}, and suppose each value
\[
V_s:=E(g_s)
\]
is finite. Then
\begin{equation}
\label{eq:poisson-gate}
V=r+PV,
\end{equation}
and necessarily
\begin{equation}
\label{eq:zero-gain}
\Pi r=0.
\end{equation}
\end{corollary}

\proof{Proof.}
The tail of \(g_s\) is
\[
Sg_s
=
\left(
e_s^\top Pr,
e_s^\top P^2r,
\ldots
\right).
\]
Coordinatewise,
\[
Sg_s
=
\sum_{s'}P(s'\mid s)g_{s'}.
\]
Translativity and linearity therefore give
\[
V_s
=
r_s+E(Sg_s)
=
r_s+\sum_{s'}P(s'\mid s)V_{s'}.
\]
Stacking these identities over states yields \eqref{eq:poisson-gate}. Since \(\Pi P=\Pi\),
\[
\Pi(I-P)=0.
\]
Multiplying \((I-P)V=r\) by \(\Pi\) gives \eqref{eq:zero-gain}.
\Halmos\endproof

Condition \eqref{eq:zero-gain} is only necessary for an arbitrary linear translative evaluation; no general sufficiency claim is made here. The equation is the familiar finite-state Poisson compatibility condition, and its relation to gain, bias, higher-order bias, and generalized inverses is classical in finite-state Markov decision processes \citep{lamond1989generalized,cao2008bias}. In Section~\ref{sec:abel}, Abel summation will provide a positive benchmark for which this zero-gain condition is also sufficient for finiteness of the statewise Abel limits.

\subsection{Interpretation}

Proposition~\ref{thm:translation-necessity} is a supporting bridge rather than a second headline obstruction theorem. It separates analytic extension from scalar recursive sufficiency: extending a series beyond ordinary convergence is one operation, whereas using the resulting scalar as a one-step continuation state is an additional structural claim. Under finite-support regularity, that claim forces the classical additive prefix law. Monotonicity then excludes persistent nonnegative divergence, while the finite-state version imposes the zero-gain condition \(\Pi r=0\).

These results are not a characterization of all dynamically consistent infinite-horizon objectives. Their role is to provide the summability-specific baseline against which the Ramanujan domain obstruction and the Abel positive benchmark can be compared.

\section{Literal Candelpergher Ramanujan summation}
\label{sec:ramanujan}

We now consider a generalized summation procedure for which value assignment and the structural prerequisites for Bellman recursion separate sharply. This section applies the same common-continuation substitution used in the weighted analysis, but the failure occurs one gate earlier: before asking how the continuation can be compressed, we must first be able to prefix the same continuation by each feasible current alternative. The main decision-theoretic result of this section is Theorem~\ref{thm:ram-domain-obstruction}: a one-step policy deviation can leave the literal analytic sample domain even when the continuation is held fixed. Carlson uniqueness and the tail-rigidity lemma are the classical-input-to-sequential-obstruction bridge used to establish that result, rather than an independent claim of new analytic summability theory. Throughout this section, \emph{Ramanujan summation} means the \emph{literal Candelpergher Ramanujan summation} of Candelpergher~\cite{candelpergher2017}. No result is asserted for other Ramanujan-type methods, other analytic regularizations, or arbitrary sequence-level extensions.

\subsection{Literal analytic domain and shift law}

Candelpergher first fixes the analytic growth class used by the construction. In his notation, a function analytic in a right half-plane \(\operatorname{Re}z>a\) is of exponential type \(<\alpha\) if there exist constants \(C>0\) and \(\beta<\alpha\) such that
\[
|f(z)|\le C e^{\beta |z|}
\]
throughout that half-plane. The class \(O^\alpha\) consists of functions satisfying such a bound in some half-plane \(\operatorname{Re}z>a\) with \(a<1\) \citep[Def.~1]{candelpergher2017}. The literal Ramanujan construction used here is restricted to \(O^\pi\).

\begin{keydefinition}[Literal Candelpergher Ramanujan functional]
For \(f\in O^\pi\), let \(R_f\in O^\pi\) be the unique function satisfying
\begin{equation}
\label{eq:ram-difference}
R_f(x)-R_f(x+1)=f(x),
\qquad
\int_1^2 R_f(x)\,dx=0.
\end{equation}
The literal Candelpergher Ramanujan value is
\begin{equation}
\label{eq:ram-definition}
\mathcal R(f):=R_f(1).
\end{equation}
Equivalently, Candelpergher writes this as the Ramanujan sum of the positive-integer samples \(f(1),f(2),\ldots\) \citep[Def.~2]{candelpergher2017}.
\end{keydefinition}

For the decision-theoretic statements below, restrict the literal construction to the real-valued subclass
\[
O^\pi_{\mathbb R}
=
\{f\in O^\pi: f(x)\in\mathbb R\text{ for every real }x\text{ in its domain}\}.
\]
If \(f\in O^\pi_{\mathbb R}\), then analyticity and the identity theorem give \(f(z)=\overline{f(\bar z)}\) on the relevant right half-plane. Hence \(\widetilde R(z)=\overline{R_f(\bar z)}\) satisfies the same difference equation and the same real normalization integral as \(R_f\). Uniqueness of the normalized remainder therefore gives \(\widetilde R=R_f\), so \(R_f\) is real on the real axis and \(\mathcal R(f)\in\mathbb R\).

The indexing in this paper is therefore fixed as follows. Associate to \(f\in O^\pi_{\mathbb R}\) the stream
\begin{equation}
\label{eq:ram-sample-stream}
u^f_t=f(t+1),
\qquad t\in\Nzero,
\end{equation}
so that
\[
u^f=(f(1),f(2),f(3),\ldots).
\]
The induced literal real sample domain is
\begin{equation}
\label{eq:ram-domain}
D_{\mathcal R}
=
\{u^f:f\in O^\pi_{\mathbb R}\}
\subseteq\mathbb R^{\Nzero}.
\end{equation}
Carlson uniqueness below implies that each stream in \(D_{\mathcal R}\) has at most one admissible interpolation in \(O^\pi_{\mathbb R}\). Thus the literal construction induces a well-defined partial real stream evaluation on \(D_{\mathcal R}\); we use the same symbol \(\mathcal R\) for this induced evaluation. Because the analytic interpolation is part of the literal construction, membership in \(D_{\mathcal R}\) is stronger than merely specifying a bare sequence of real numbers.

For an analytic function, write
\[
(Sf)(x)=f(x+1).
\]
Its positive-integer sample stream is the left shift of \(u^f\). Candelpergher's normalized functional obeys the exact shift law
\begin{equation}
\label{eq:ram-shift}
\mathcal R(Sf)
=
\mathcal R(f)-f(1)+\int_1^2 f(x)\,dx,
\end{equation}
which is Eq.~(2.1) in Candelpergher~\cite{candelpergher2017}. In particular, defining the additive one-step defect by
\[
\Delta_{\mathcal R}(f)
=
\mathcal R(f)-f(1)-\mathcal R(Sf),
\]
we have
\begin{equation}
\label{eq:ram-defect}
\Delta_{\mathcal R}(f)
=
-\int_1^2 f(x)\,dx.
\end{equation}
Equation~\eqref{eq:ram-shift} is a source property of the literal summation rule, not a novelty claim of this paper. Candelpergher explicitly notes that it differs from the usual translation law of convergent series \citep[Sec.~2.1.1]{candelpergher2017}.

\subsection{Tail rigidity of the literal sample domain}

The key analytic input is the Carlson uniqueness theorem in the form used by Candelpergher. His Appendix~B states that if a function is analytic in a half-plane \(\operatorname{Re}z>-d\), with \(0<d<1\), and satisfies a bound
\[
|f(z)|\le a e^{b|z|}
\qquad
\text{for some }a>0,
\quad b<\pi,
\]
then vanishing at every nonnegative integer forces the function to vanish identically \citep[App.~B]{candelpergher2017}. This is the Carlson uniqueness regime used below.

\begin{lemma}[Ramanujan tail rigidity]
\label{lem:ram-tail-rigidity}
Let \(f,g\in O^\pi\). If there exists an integer \(N\ge1\) such that
\[
f(n)=g(n)
\qquad
\text{for every integer }n\ge N,
\]
then
\[
f=g.
\]
\end{lemma}

\proof{Proof.}
Let
\[
h=f-g.
\]
Then \(h\in O^\pi\) and
\[
h(n)=0
\qquad
\text{for every }n\ge N.
\]
By the definition of \(O^\pi\), there is some \(a<1\) such that \(h\) is analytic in \(\operatorname{Re}z>a\), and there are \(C>0\) and \(\beta<\pi\) such that
\[
|h(z)|\le C e^{\beta |z|}
\]
there. Define
\[
q(z)=h(z+N).
\]
Then \(q\) is analytic in \(\operatorname{Re}z>a-N\). Since \(a-N<0\), choose
\[
0<d<\min\{1,N-a\}.
\]
The half-plane \(\operatorname{Re}z>-d\) lies inside the domain of analyticity of \(q\). Moreover,
\[
|q(z)|
=
|h(z+N)|
\le
C e^{\beta |z+N|}
\le
C e^{\beta N}e^{\beta |z|},
\]
so \(q\) satisfies the same strict type bound \(\beta<\pi\). Finally,
\[
q(m)=h(m+N)=0
\qquad
m=0,1,2,\ldots.
\]
The Carlson theorem quoted above therefore yields \(q\equiv0\). Hence \(h\equiv0\) on a right half-plane and thus \(f=g\).
\Halmos\endproof

The uniqueness theorem is classical; the following sequential consequences are the ones used here.

\begin{keycorollary}[No nonzero finite-support literal sample stream]
If \(u\in D_{\mathcal R}\) is eventually zero, then \(u=0\). In particular,
\[
(1,0,0,\ldots)\notin D_{\mathcal R}.
\]
\end{keycorollary}

\proof{Proof.}
Write \(u=u^f\) for some \(f\in O^\pi_{\mathbb R}\). If \(u\) is eventually zero, then \(f(n)=0\) for all sufficiently large positive integers. Lemma~\ref{lem:ram-tail-rigidity}, applied to \(f\) and the zero function, gives \(f=0\), hence \(u=0\).
\Halmos\endproof

\begin{keycorollary}[Unique admissible finite prefix]
Fix a literal Ramanujan-admissible tail \(u=(u_0,u_1,\ldots)\) and a prefix length \(m\). There is at most one prefix
\[
(c_0,\ldots,c_{m-1})
\]
for which
\[
(c_0,\ldots,c_{m-1},u_0,u_1,\ldots)
\in D_{\mathcal R}.
\]
In particular,
\[
c\oplus u,\ d\oplus u\in D_{\mathcal R}
\quad\Longrightarrow\quad
c=d.
\]
\end{keycorollary}

\proof{Proof.}
If two admissible streams share the same infinite tail after a fixed finite prefix, their corresponding \(O^\pi\) interpolants agree at all sufficiently large positive integers. Lemma~\ref{lem:ram-tail-rigidity} makes the interpolants identical, so all prefix samples must coincide.
\Halmos\endproof

\subsection{A one-state policy-prefix domain obstruction}

The preceding corollary has an immediate controlled-MDP consequence. Under the literal evaluation, a one-step Bellman-style comparison between two current-action alternatives sharing the same continuation requires both prefixed streams to lie in its domain. Tail rigidity makes that elementary substitution fail for the literal sample domain.

\begin{theorem}[One-state policy-prefix domain obstruction]
\label{thm:ram-domain-obstruction}
Consider the one-state deterministic self-loop MDP under a policy class containing deterministic nonstationary policies, with actions \(a_0,a_1\) and rewards
\[
r(a_0)=0,
\qquad
r(a_1)=1.
\]
Let \(\pi_0\) choose \(a_0\) forever. Its expected stage-reward stream
\[
u=(0,0,0,\ldots)
\]
belongs to \(D_{\mathcal R}\). If instead one chooses \(a_1\) once and then follows the same continuation \(\pi_0\), the resulting stream is
\[
v=(1,0,0,\ldots),
\]
and
\[
v\notin D_{\mathcal R}.
\]
Consequently, literal Candelpergher Ramanujan admissibility is not closed under a one-step action deviation followed by a common continuation.
\end{theorem}

\proof{Proof.}
The zero stream is generated by the zero function, which lies in \(O^\pi_{\mathbb R}\), so \(u\in D_{\mathcal R}\). The deviating stream \(v\) is nonzero and has finite support. The preceding finite-support corollary therefore gives
\[
v\notin D_{\mathcal R}.
\]
The two policies have exactly the same continuation after the first action, so the failure is precisely a failure of one-step policy-prefix closure.
\Halmos\endproof

\begin{keyremark}[Domain failure, not an infinite-memory theorem]
Theorem~\ref{thm:ram-domain-obstruction} occurs before a realization-dimension question is well posed for the full policy class. The literal objective is not closed under the control operation itself. The conclusion is therefore \emph{domain incompatibility}; it does not state that Ramanujan summation requires an infinite-dimensional Bellman state, nor that every Ramanujan-type method is incompatible with dynamic programming. Nor does the theorem rule out constructing a larger sequence-level domain and an extension that agrees with the literal functional on \(D_{\mathcal R}\) while assigning values to additional streams. It states only that the literal Candelpergher admissible domain itself is not closed under this elementary control operation.
\end{keyremark}

The control significance is stronger than the observation that one analytic sequence is inadmissible. In the one-state model, the two policies differ only in the current action and share the entire future continuation. A Bellman-style action comparison would therefore need to place both prefixed streams under the same evaluation rule, yet the literal domain contains only one of them. The obstruction occurs before a recursive state, scalar backup, or optimality operator can be asked to compare those actions.

This also explains why restricted fixed-policy evaluation and controlled policy branching must be separated. A structured fixed policy can generate an analytically admissible stream even though an ordinary one-step deviation from the same continuation leaves the literal domain.

\subsection{The scalar Ramanujan value is not a sufficient one-step state}

The domain obstruction is the earliest failure for the full controlled policy class. Independently, even on the analytic function class where the literal functional is defined, the current sample and scalar tail value do not determine the current literal value.

\begin{proposition}[Scalar Ramanujan value insufficiency]
\label{prop:ram-scalar-insufficiency}
There is no universal map
\[
W:\mathbb{R}^2\to\mathbb{R}
\]
such that
\begin{equation}
\label{eq:ram-scalar-recursion}
\mathcal R(f)
=
W\bigl(f(1),\mathcal R(Sf)\bigr)
\end{equation}
for every \(f\in O^\pi_{\mathbb R}\).
\end{proposition}

\proof{Proof.}
Linearity follows from uniqueness of the normalized remainder. The explicit polynomial witness $h(x)=5x^2-14x+9$ satisfies
\[
h(1)=0,\qquad \mathcal R(Sh)=0,\qquad \mathcal R(h)=\tfrac13,
\]
whereas the zero function has the same input pair $(0,0)$ and value $0$. The elementary calculations are given in Electronic Companion~EC.2.
\Halmos\endproof

Proposition~\ref{prop:ram-scalar-insufficiency} should not be conflated with Theorem~\ref{thm:ram-domain-obstruction}. The theorem says that the standard controlled policy class may leave the literal domain. The proposition says that, even within the analytic domain, one scalar continuation value is insufficient for the literal shift relation.

\subsection{Stagewise monotonicity fails}

A second independent issue concerns reward dominance.

\begin{proposition}[Stagewise monotonicity failure]
\label{prop:ram-monotonicity}
The induced real stream evaluation \(\mathcal R:D_{\mathcal R}\to\mathbb R\) is not stagewise monotone on \(D_{\mathcal R}\). In particular, there exists \(f\in O^\pi_{\mathbb R}\) such that
\[
f(n)\ge0
\qquad
\text{for every integer }n\ge1,
\]
but
\[
\mathcal R(f)<0=\mathcal R(0).
\]
Moreover, the corresponding nonnegative sample stream is realizable as the expected stage-reward stream of a one-state two-action finite MDP under a nonstationary randomized policy.
\end{proposition}

\proof{Proof.}
Take $f(x)=e^{-(x-1)}-e^{-2(x-1)}$. Its positive-integer samples are nonnegative, while Candelpergher's exponential formula gives $\mathcal R(f)<0$. In the one-state self-loop MDP with rewards $0$ and $1$, choosing reward $1$ at time $t$ with probability $e^{-t}-e^{-2t}$ realizes the same expected-reward stream. Details are in Electronic Companion~EC.2.
\Halmos\endproof

The proposition is a decision-theoretic warning rather than a statement about domain closure: maximizing the literal functional need not respect coordinatewise improvements in expected stage rewards.

\subsection{Restricted fixed-policy calibration}

The controlled-domain obstruction does not preclude tractable evaluation on restricted fixed-policy families. Electronic Companion~EC.1 derives an explicit spectral formula for literal Ramanujan values when a finite Markov reward process has transition spectrum in $(0,1]$. The formula is expressed through the stationary projection, group inverse, and principal matrix logarithm, thereby calibrating the literal value against standard finite-state gain/bias objects \citep{lamond1989generalized,cao2008bias}. This result is intentionally supporting: fixed-policy tractability is weaker than closure under policy branching and does not repair Theorem~\ref{thm:ram-domain-obstruction}.

\subsection{Interpretation}

Literal Candelpergher evaluation can be tractable on restricted fixed-policy spectral families, yet its analytic sample domain is tail-rigid under control; within that domain, the scalar value is also insufficient for one-step recursion and stagewise monotonicity fails. The main controlled-MDP conclusion is therefore Theorem~\ref{thm:ram-domain-obstruction}. Its significance is not that Ramanujan summation has an unusual analytic domain in isolation, but that the domain excludes an elementary policy operation: changing the current action while holding the continuation fixed. Domain failure can therefore precede both Bellman optimality and realization-dimension questions. This conclusion is specific to the literal Candelpergher construction.

\section{Abel summation as a positive benchmark}
\label{sec:abel}

Literal Candelpergher Ramanujan summation can fail before recursion is posed because its analytic sample domain is not freely prefix closed. Abel summation shows that this is not a generic feature of generalized summation. For bounded streams define
\[
D_A=\left\{x\in\ell^\infty:\lim_{\alpha\uparrow1}\sum_{t\ge0}\alpha^t x_t\ \text{exists finitely}\right\},
\qquad
\mathcal A(x)=\lim_{\alpha\uparrow1}\sum_{t\ge0}\alpha^t x_t.
\]
Abel-type criteria also have a classical role in undiscounted dynamic programming \citep{blackwell1962discrete,veinott1966nodiscount,hosaka1999nondiscounted,jonsson2023axiomatic}.

\begin{proposition}[Abel translative compatibility]
\label{prop:abel-compatibility}
$D_A$ is a real vector space closed under finite prefixing. On $D_A$, $\mathcal A$ is linear, agrees with ordinary summation when the latter converges, is stagewise monotone, and satisfies
\[
\mathcal A(c\oplus x)=c+\mathcal A(x),
\qquad
\mathcal A\!\left(c\oplus\sum_i p_i x^i\right)=c+\sum_i p_i\mathcal A(x^i).
\]
Thus Abel summation is algebraically scalar Bellman-compatible on its finite-valued domain.
\end{proposition}

\proof{Proof.}
For $F_x(\alpha)=\sum_{t\ge0}\alpha^t x_t$, absolute convergence for every $\alpha<1$ gives linearity, and $F_{c\oplus x}(\alpha)=c+\alpha F_x(\alpha)$ gives the prefix law after $\alpha\uparrow1$. Abel's theorem gives ordinary regularity \citep{hardy1949divergent}. If $x\le y$, then $F_{y-x}(\alpha)\ge0$ for all $\alpha<1$; taking the finite limit gives monotonicity.
\Halmos\endproof

Grandi's stream $(1,-1,1,-1,\ldots)$ is ordinarily divergent but satisfies
\[
\mathcal A(1,-1,1,-1,\ldots)=\lim_{\alpha\uparrow1}\frac1{1+\alpha}=\frac12.
\]
It is generated by the two-state periodic Markov reward process with $P=\bigl(\begin{smallmatrix}0&1\\1&0\end{smallmatrix}\bigr)$ and $r=(1,-1)^\top$. Thus ordinary divergence can coexist with a finite translative evaluation.

\subsection{Finite-state zero-gain characterization}

For a finite stochastic matrix $P$, reward vector $r$, and $0\le\alpha<1$, let
\[
V_\alpha=(I-\alpha P)^{-1}r,
\]
and let $\Pi$ be the spectral projection onto the eigenvalue-$1$ subspace.

\begin{proposition}[Finite-state Abel zero-gain characterization]
\label{prop:abel-finite-state}
The limit $\lim_{\alpha\uparrow1}V_\alpha$ exists finitely if and only if $\Pi r=0$. In that case
\[
V_{\mathcal A}=(I-P)^\#r.
\]
\end{proposition}

\proof{Proof.}
Because a finite stochastic matrix is power bounded, eigenvalue $1$ is semisimple and $\mathbb C^m=\operatorname{im}\Pi\oplus\ker\Pi$. On $\operatorname{im}\Pi$,
\[
\Pi V_\alpha=(1-\alpha)^{-1}\Pi r,
\]
so a finite limit requires $\Pi r=0$. On $\ker\Pi$, $1$ is absent from the spectrum, hence $(I-\alpha P)^{-1}r\to(I-P|_{\ker\Pi})^{-1}r$. Relative to $\operatorname{im}\Pi\oplus\ker\Pi$ we have
\[
I-P=0\oplus\bigl(I-P|_{\ker\Pi}\bigr),
\]
with the second block invertible. The defining equations of the group inverse therefore give
\[
(I-P)^\#=0\oplus\bigl(I-P|_{\ker\Pi}\bigr)^{-1},
\]
so the limit is $(I-P)^\#r$; see \citet{lamond1989generalized} for the finite-state Markov-chain background. Further details are in Electronic Companion~EC.3.
\Halmos\endproof

\subsection{Topology caveat}

The proposition is algebraic, not a claim of product-topology continuity. If $e_n$ has a single $1$ in coordinate $n$, then $e_n\to0$ in the product topology but $\mathcal A(e_n)=1$ for every $n$. Hence $\mathcal A$ is not product-continuous on its partial domain and should not be placed in $\Fclass_{\prodtop}^{\aut}$ under Section~\ref{sec:setup}'s continuity convention.

Abel summation is therefore a positive benchmark: its finite-valued domain is prefix closed, its scalar value is translative, and finite Markov reward processes have finite statewise Abel values exactly on the zero-gain class. This isolates the Ramanujan obstruction as method-specific rather than an unavoidable consequence of generalized summation.

\section{Related work and discussion}
\label{sec:discussion}

\subsection{Relation to realization and recursive-state theories}

Infinite expected-reward streams support axiomatic valuations beyond standard discounting \citep{jonsson2023axiomatic}. Finitely additive and related extensions provide another direction \citep{neyman2023valuations,flesch2026finitelyadditive}. Adjacent recursive-state theories operate on different objects. Bellman-closed statistics summarize return distributions \citep{rowland2019statistics}; exact finite-horizon return functionals form another class \citep{marthe2023beyond}; and preference-augmented states handle history-dependent objectives \citep{haskell2026history}. Here the earlier question is whether one-step control substitutions remain in the domain of a specified infinite stream functional, and what continuation information exact recursion requires.

For weighted criteria, finite Hankel rank, recurrence/rationality, and minimal linear/state-affine realization are classical \citep{schutzenberger1961automata,carlylepaz1971realizations,fliess1974hankel,sontag1979realization}. Analytic nonlinear minimality, topological automata, structured nonlinear rank-type lower bounds, and finite-dimensional information-state theory are also established \citep{sussmann1976minimal,ehrig1974topological,petreczky2008semialgebraic,charalambous2000information}; related canonicalization and structured state-space work continues this line \citep{grigoryeva2021canonicalization,petreczky2026stability}. Theorem~\ref{thm:weighted-hankel-characterization} occupies a specific gap: it keeps the ordinary temporal Hankel invariant, weakens the competing Euclidean encoder to continuity alone, and asks for exact autonomous continuation state under repeated one-hole updates. It is therefore not a general nonlinear input-output realization theorem, nor does it address arbitrary topological state spaces or approximation.

Sontag's bounded discrete-time realization theory is especially close: it relates Hankel/behavior information to minimal state-linear dimension and shows, under analytic, polynomial, and related regularity hypotheses, when more general realizations force state-affine finite realizability \citep{sontag1979realization}. This prevents interpreting our theorem as a claim that Hankel methods were previously confined to linear competitors. The distinction here is narrower: the lower bound assumes only continuity of the competing Euclidean encoder, and its proof is topological rather than differential or algebraic. Petreczky and Vidal likewise obtain nonlinear state-dimension lower bounds from an algebraic Hankel rank defined through Krull dimension for semi-algebraic and implicit-polynomial systems \citep{petreczky2008semialgebraic}; both the invariant and competitor class differ from those used here. Grigoryeva and Ortega provide a nearby quotient-state perspective for recurrent filters and emphasize the difficulty of geometric dimension questions for nonlinear canonical states \citep{grigoryeva2021canonicalization}.

Predictive-state/Hankel representations also appear in controlled stochastic systems, where the rank describes environment or predictive dynamics \citep{bacon2015timing}. Here $H_w$ belongs to the \emph{criterion}: Corollary~\ref{cor:criterion-memory-cost} shows that the same physical one-state MDP can require different endogenous continuation dimensions solely because the temporal objective changes.

A related comparison comes from information-based complexity. For linear problems, Kacewicz and Wasilkowski show that continuous nonlinear information is no more powerful than linear information in the worst-case approximation setting, while the average-case conclusion differs \citep{kacewicz1986continuous}. Their result concerns nonrecursive information acquisition and approximation error. Theorem~\ref{thm:weighted-hankel-characterization}, by contrast, concerns exact autonomous recursive continuation states under repeated one-hole updates and identifies the resulting dimension with the temporal Hankel rank.

\subsection{One substitution test, several failure levels}

The two main axes are unified by one operation: keep a continuation fixed, vary the current prefix or branch, and ask whether the alternatives can still be evaluated through a common recursive state. Literal Ramanujan fails the \emph{domain gate}: the two alternatives need not share an admissible domain. Weighted $\ell^1$ criteria are total on the bounded cube, so the next question is representation; temporal Hankel rank gives its exact continuous autonomous dimension. Quasi-hyperbolic weighting then shows that a finite continuation state need not collapse to the scalar value, and alternating-geometric weighting shows that scalar recursion need not preserve contextual order. Proposition~\ref{thm:translation-necessity} and Abel summation supply, respectively, the scalar summability bridge and a positive benchmark.

This hierarchy separates \emph{Layer I} structural compatibility from \emph{Layer II} optimal-control analysis. The present results do not establish optimal-policy existence, stationary reduction, Bellman optimality, contraction, or algorithmic convergence. Those questions begin only after the relevant evaluation domain and recursive state are well posed; recursive-utility Bellman results under model-specific assumptions provide one downstream example \citep{jaskiewicz2026bellman}.

\subsection{Scope and limitations}

The evaluated object is the deterministic stream of expected stage rewards. The weighted theorem concerns exact continuous autonomous Euclidean summaries; it does not exclude discontinuous or infinite-dimensional encodings, approximate compression, or explicit calendar-time augmentation. Its rank is the temporal shift rank of the weight sequence in the present affine one-hole architecture and should not be identified without proof with a different Hankel invariant. For literal Ramanujan summation, the $O^\pi$ domain, normalized difference equation, shift law, and Carlson uniqueness are source facts from Candelpergher \citep{candelpergher2017}; the paper-specific consequence is the controlled common-continuation domain obstruction. The supporting fixed-policy spectral calculations are placed in the electronic companion precisely to keep fixed-policy tractability separate from policy-branching compatibility.

\section{Conclusion}
\label{sec:conclusion}

A finite value for an infinite expected-reward stream does not by itself justify Bellman recursion. One-step control contexts must first preserve the evaluation domain; only then do exact autonomous realization, scalar sufficiency, and contextual order become meaningful.

For absolutely summable weighted criteria, temporal Hankel rank gives an exact pre-Bellman state-augmentation test in the stated architecture. Finite rank yields a constructive fixed-dimensional autonomous state; infinite rank rules out every fixed finite-dimensional continuous Euclidean encoder. Equivalently, criterion-induced temporal nonstationarity can be internalized into finite endogenous memory exactly in the finite-rank case. Because the lower bound already holds in a one-state two-action MDP, this memory requirement is attributable to the objective rather than to physical-state complexity. When the rank is infinite, exact dynamic programming can still use an external clock, an infinite-dimensional state, a discontinuous encoding, or approximation; the theorem rules out only fixed finite-dimensional continuous autonomous compression.

Literal Candelpergher Ramanujan summation fails at an earlier gate: two current actions sharing the same continuation need not both lie in the literal evaluation domain. Abel summation shows that such domain failure is not generic to generalized summation. Passing these structural gates remains prior to, and does not by itself establish, optimal-policy existence, stationary optimality, contraction, or algorithmic convergence.

The central conclusion is therefore structural rather than algorithmic:
\[
\boxed{\text{finite generalized value assignment}\;\not\Rightarrow\;\text{Bellman composability}.}
\]

% Detailed supporting proofs and calibrations are provided in the electronic companion.

% Retain the EC's original independent equation and proposition numbering.
% The printed page numbers continue through the complete arXiv document.
\setcounter{section}{0}
\setcounter{subsection}{0}
\setcounter{equation}{0}
\setcounter{theorem}{0}
\setcounter{lemma}{0}
\setcounter{proposition}{0}
\setcounter{corollary}{0}
\setcounter{definition}{0}
\setcounter{remark}{0}
\renewcommand{\theproposition}{EC.\arabic{proposition}}
\renewcommand{\theHequation}{EC.\arabic{equation}}
\renewcommand{\theHsection}{EC.\arabic{section}}
\renewcommand{\theHsubsection}{EC.\arabic{section}.\arabic{subsection}}
\renewcommand{\theHtheorem}{EC.\arabic{section}.\arabic{theorem}}
\renewcommand{\theHlemma}{EC.\arabic{section}.\arabic{lemma}}
\renewcommand{\theHproposition}{EC.\arabic{proposition}}
\renewcommand{\theHcorollary}{EC.\arabic{section}.\arabic{corollary}}
\renewcommand{\theHdefinition}{EC.\arabic{section}.\arabic{definition}}
\renewcommand{\theHremark}{EC.\arabic{section}.\arabic{remark}}
% Scope EC cross-reference keys locally: one key also occurs in the main paper.
\let\arxivOriginalLabel\label
\let\arxivOriginalRef\ref
\renewcommand{\label}[1]{\arxivOriginalLabel{EC:#1}}
\renewcommand{\ref}[1]{\arxivOriginalRef{EC:#1}}
\makeatletter
\let\ltx@label\label
\makeatother
\TITLE{Electronic Companion to: When Can an Infinite-Stream Evaluation Support Dynamic Programming? Compatibility, Recursive Compression, and Generalized Summation}
\ARTICLEAUTHORS{\AUTHOR{Yingzhe Xu}\AFF{The Australian National University}}
\ABSTRACT{This electronic companion contains supporting derivations moved from the main manuscript for review readability. It gives the full fixed-policy Ramanujan spectral calibration, the elementary algebra behind the two Ramanujan witness propositions, and the full finite-state Abel zero-gain derivation. No additional headline theorem is introduced here.}
\maketitle

\section*{EC.1. Fixed-policy Ramanujan spectral calibration}
This section supplies the full derivation supporting the restricted fixed-policy calibration summarized in the main paper.

\begin{proposition}[Fixed-policy Ramanujan spectral calibration]
\label{prop:ram-fixed-policy}
Let \(P\in\mathbb{R}^{m\times m}\) be a finite stochastic matrix whose every eigenvalue is real and lies in \((0,1]\).
Let \(r\in\mathbb{R}^m\), and let
\[
\Pi
=
\lim_{N\to\infty}
\frac1N\sum_{n=0}^{N-1}P^n
\]
be the Ces\`aro, equivalently spectral, projection onto the eigenvalue-\(1\) subspace. Define
\[
g=\Pi r,
\qquad
q=(I-\Pi)r,
\qquad
H=(I-P)^\#,
\]
and let
\[
C=\ker\Pi,
\qquad
Q=P|_C.
\]
Then \(\sigma(Q)\subset(0,1)\), so the real principal matrix logarithm
\[
L=\Log Q
\]
is well defined and nonsingular \citep{higham2008functions}. The vector of componentwise literal Candelpergher Ramanujan values of the statewise streams
\[
r,\ Pr,\ P^2r,\ldots
\]
is
\begin{equation}
\label{eq:ram-fixed-policy}
V_R
=
\frac12 g
+Hq
+L^{-1}q.
\end{equation}
Equivalently, since \(H\Pi=0\),
\[
V_R
=
\frac12 g
+h
+(\Log Q)^{-1}(r-g),
\qquad
h=Hr.
\]
\end{proposition}

This appendix proves the proposition stated above without a diagonalizability assumption on the transient spectral part. The argument also verifies explicitly that the matrix-valued interpolation and its normalized remainder belong componentwise to Candelpergher's admissible analytic class. We use standard facts about the principal matrix logarithm and matrix functions \citep{higham2008functions}, together with the group-inverse decomposition for finite Markov chains \citep{lamond1989generalized}.

Let
\[
\Pi
=
\lim_{N\to\infty}\frac1N\sum_{n=0}^{N-1}P^n,
\qquad
g=\Pi r,
\qquad
q=(I-\Pi)r,
\]
and set
\[
C=\ker\Pi,
\qquad
Q=P|_C,
\qquad
L=\Log Q.
\]
Because a finite stochastic matrix is power bounded, every eigenvalue on the unit circle is semisimple. Under the assumptions of the proposition stated above, the only unit-circle eigenvalue is therefore \(1\), and the decomposition
\[
\mathbb R^m
=
\operatorname{im}\Pi\oplus C
\]
is \(P\)-invariant. On \(\operatorname{im}\Pi\), \(P\) acts as the identity; on \(C\), it acts as \(Q\), whose spectrum is contained in \((0,1)\). Hence \(L\) is the real principal logarithm of \(Q\), \(e^L=Q\), and \(L\) is nonsingular.

Define the vector-valued interpolation
\begin{equation}
\label{eq:appendix-F}
F(z)
=
g+e^{(z-1)L}q.
\end{equation}
For every positive integer \(n\),
\[
F(n)
=
g+Q^{n-1}q
=
P^{n-1}r.
\]
Thus \(F\) interpolates exactly the statewise expected-reward sequence used in the literal Ramanujan evaluation.

We first verify admissibility. Put \(L=VJV^{-1}\) over \(\mathbb C\), with Jordan blocks
\[
J_\mu=\mu I+N,
\qquad
\mu=\log\lambda<0,
\qquad
N^d=0,
\]
where \(0<\lambda<1\). For a block of size \(d\),
\begin{equation}
\label{eq:jordan-exp}
e^{(z-1)J_\mu}
=
e^{\mu(z-1)}
\sum_{j=0}^{d-1}
\frac{(z-1)^j}{j!}N^j.
\end{equation}
Consequently every entry of \(e^{(z-1)L}\), and hence every component of \(F(z)\), is a finite linear combination of terms of the form
\[
p(z)e^{\mu(z-1)},
\qquad
\mu<0,
\]
with \(p\) a polynomial. On any right half-plane \(\operatorname{Re}z>a\) with \(a<1\), the exponential factor is bounded by a constant depending on \(a\), while the polynomial factor has exponential type zero. Hence every component is analytic there and has exponential type strictly below \(\pi\). Thus the components of \(F\) lie in the literal \(O^\pi\) class used by Candelpergher \citep{candelpergher2017}. Moreover, \(g\), \(q\), and \(L\) are real, so \(F(x)\) is real for real \(x\); componentwise, the interpolation therefore lies in the real subclass \(O^\pi_{\mathbb R}\) used for the decision-theoretic stream domain in the main paper. This argument covers nontrivial Jordan blocks and does not require \(Q\) to be diagonalizable.

We now construct the normalized fractional remainder directly. For the eigenvalue-\(1\) component define
\[
R_1(z)
=
\left(\frac32-z\right)g.
\]
Then
\[
R_1(z)-R_1(z+1)=g,
\qquad
\int_1^2R_1(x)\,dx=0.
\]
For the complementary component define
\begin{equation}
\label{eq:appendix-RC}
R_C(z)
=
e^{(z-1)L}(I-Q)^{-1}q
+L^{-1}q.
\end{equation}
The same Jordan-block calculation as in \eqref{eq:jordan-exp} shows that every component of \(R_C\) lies in \(O^\pi\). Because all matrices and vectors in \eqref{eq:appendix-RC} are real, \(R_C(x)\) is real for real \(x\); hence the components lie in \(O^\pi_{\mathbb R}\). The same is immediate for \(R_1\), and therefore for \(R_F=R_1+R_C\). Since \(e^L=Q\),
\begin{align*}
R_C(z)-R_C(z+1)
&=
e^{(z-1)L}
\bigl[(I-Q)^{-1}-Q(I-Q)^{-1}\bigr]q\\
&=
e^{(z-1)L}q.
\end{align*}
Moreover, because \(L\) is nonsingular,
\[
\int_1^2e^{(x-1)L}\,dx
=
\int_0^1e^{tL}\,dt
=
L^{-1}(e^L-I)
=
L^{-1}(Q-I).
\]
Therefore
\begin{align*}
\int_1^2R_C(x)\,dx
&=
L^{-1}(Q-I)(I-Q)^{-1}q+L^{-1}q\\
&=
-L^{-1}q+L^{-1}q\\
&=0.
\end{align*}

Set
\[
R_F(z)=R_1(z)+R_C(z).
\]
Combining the preceding identities gives
\[
R_F(z)-R_F(z+1)
=
g+e^{(z-1)L}q
=
F(z),
\]
and
\[
\int_1^2R_F(x)\,dx=0.
\]
By the uniqueness in Candelpergher's literal definition, \(R_F\) is the normalized fractional remainder of the interpolation \(F\). Hence the componentwise literal Ramanujan value is
\begin{align*}
V_R
=R_F(1)
&=
\frac12g+(I-Q)^{-1}q+L^{-1}q.
\end{align*}

It remains only to identify the group-inverse term. Because the eigenvalue \(1\) of \(P\) is semisimple,
\[
I-P
=
0\oplus(I-Q)
\]
relative to \(\operatorname{im}\Pi\oplus C\). Therefore its group inverse is
\[
H=(I-P)^\#
=
0\oplus(I-Q)^{-1}.
\]
In particular,
\[
H\Pi=0,
\qquad
Hq=(I-Q)^{-1}q.
\]
Substituting into the preceding expression yields
\[
V_R
=
\frac12g+Hq+L^{-1}q,
\]
which is equation~\eqref{eq:ram-fixed-policy}. Since \(q=r-g\) and \(Hr=Hq\), the equivalent form
\[
V_R
=
\frac12g+Hr+(\Log Q)^{-1}(r-g)
\]
follows as well. This completes the proof of the proposition stated above.

\section*{EC.2. Algebra for the Ramanujan witness propositions}
\subsection*{EC.2.1. Scalar-value insufficiency calculation}

For polynomials, the normalized fractional remainders
\[
R_1(x)=\frac32-x,
\qquad
R_x(x)=-\frac{x^2}{2}+\frac{x}{2}+\frac5{12},
\]
and
\[
R_{x^2}(x)=-\frac{x^3}{3}+\frac{x^2}{2}-\frac{x}{6}+\frac13
\]
satisfy $R_f(x)-R_f(x+1)=f(x)$ and have zero integral on $[1,2]$. Hence
\[
\mathcal R(1)=\frac12,
\qquad
\mathcal R(x)=\frac5{12},
\qquad
\mathcal R(x^2)=\frac13.
\]
With $h(x)=5x^2-14x+9$, one has $h(1)=0$ and $(Sh)(x)=5x^2-4x$, so linearity gives
\[
\mathcal R(Sh)=5\left(\frac13\right)-4\left(\frac5{12}\right)=0,
\]
whereas
\[
\mathcal R(h)=5\left(\frac13\right)-14\left(\frac5{12}\right)+9\left(\frac12\right)=\frac13.
\]
This is the explicit witness used in the scalar-value insufficiency proposition in the main paper.

\subsection*{EC.2.2. Monotonicity witness and finite-MDP realization}

Let
\[
f(x)=e^{-(x-1)}-e^{-2(x-1)}.
\]
Both exponential terms belong to $O^\pi$, and $f(n)\ge0$ at every positive integer. Candelpergher's exponential formula \citep[Eq.~(1.15)]{candelpergher2017} gives
\[
\mathcal R\!\left(e^{-a(x-1)}\right)=\psi(a),
\qquad
\psi(a)=\frac{1}{1-e^{-a}}-\frac1a.
\]
Therefore $\mathcal R(f)=\psi(1)-\psi(2)$. Since
\[
\psi'(a)=\frac1{a^2}-\frac{1}{4\sinh^2(a/2)}>0
\qquad(a>0),
\]
using $2\sinh(a/2)>a$, we obtain $\mathcal R(f)<0$.

The corresponding samples are realized in the one-state self-loop MDP with rewards $0$ and $1$: at time $t$, choose the reward-$1$ action with probability
\[
p_t=e^{-t}-e^{-2t}\in[0,1].
\]
Then $\mathbb E[r_t]=p_t=f(t+1)$, proving the finite-MDP part of the stagewise-monotonicity proposition in the main paper.

\section*{EC.3. Finite-state Abel zero-gain derivation}
For completeness, we record the full finite-state proof behind the compressed argument in the main paper. For a finite Markov reward process with stochastic matrix $P$ and reward vector $r$, define
\begin{equation}
\label{eq:abel-resolvent}
V_\alpha
=
\sum_{n=0}^{\infty}\alpha^nP^nr
=
(I-\alpha P)^{-1}r,
\qquad
0\le\alpha<1.
\end{equation}
Let $\Pi$ denote the spectral, equivalently Ces\`aro, projection onto the eigenvalue-$1$ subspace.

\begin{proposition}[Finite-state Abel zero-gain characterization]
\label{prop:abel-finite-state}
For a finite stochastic matrix $P$ and reward vector $r$, the statewise Abel limit
\[
\lim_{\alpha\uparrow1}V_\alpha
\]
exists finitely if and only if
\begin{equation}
\label{eq:abel-zero-gain}
\Pi r=0.
\end{equation}
When this condition holds,
\begin{equation}
\label{eq:abel-group-inverse}
V_{\mathcal A}
=
\lim_{\alpha\uparrow1}V_\alpha
=
(I-P)^\#r.
\end{equation}
\end{proposition}

\proof{Proof.}
Because $P$ is finite and stochastic, it is power bounded. Hence the eigenvalue $1$ is semisimple, and the state space decomposes into the invariant direct sum
\[
\mathbb C^m
=
\operatorname{im}\Pi
\oplus
\ker\Pi.
\]
The projection $\Pi$ commutes with $P$, and $P$ acts as the identity on $\operatorname{im}\Pi$. Applying $\Pi$ to \eqref{eq:abel-resolvent} gives
\[
\Pi V_\alpha
=
\Pi(I-\alpha P)^{-1}r
=
\frac{1}{1-\alpha}\Pi r.
\]
If $\Pi r\ne0$, this term has no finite limit as $\alpha\uparrow1$. Thus \eqref{eq:abel-zero-gain} is necessary.

Now suppose $\Pi r=0$. Then $r\in\ker\Pi$, and $1$ is not an eigenvalue of the restriction $P|_{\ker\Pi}$. Therefore
\[
(I-\alpha P)^{-1}r
\longrightarrow
(I-P|_{\ker\Pi})^{-1}r
\qquad
(\alpha\uparrow1).
\]
Relative to $\operatorname{im}\Pi\oplus\ker\Pi$,
\[
I-P=0\oplus\bigl(I-P|_{\ker\Pi}\bigr),
\]
and the second block is invertible. Hence the defining equations of the group inverse give
\[
(I-P)^\#=0\oplus\bigl(I-P|_{\ker\Pi}\bigr)^{-1};
\]
see \citet{lamond1989generalized} for the finite-state Markov-chain background. Therefore the finite limit equals
\[
(I-P)^\#r,
\]
which proves \eqref{eq:abel-group-inverse}.
\Halmos\endproof

\end{document}